\documentclass[11pt,reqno]{amsart}
\usepackage{hyperref}
\usepackage{amsfonts,mathrsfs,bbm,latexsym,rawfonts,amsmath,amssymb,amsthm,epsfig}
\usepackage{fullpage, setspace, color}
\usepackage{cases}
\newtheorem{thm}{Theorem}[section]

\newtheorem{rem}[thm]{Remark}
\newtheorem{lem}[thm]{Lemma}
\newtheorem{prop}[thm]{Proposition}
\newtheorem{defn}[thm]{Definition}

\numberwithin{equation}{section}

\newcommand{\al}{\alpha}

\newcommand{\De}{\Delta}
\newcommand{\ep}{\varepsilon}

\newcommand{\om}{\omega}
\newcommand{\Om}{\Omega}
\newcommand{\ga}{\gamma}

\newcommand{\g}{\mathfrak{g}}

\newcommand{\U}{\mathbb{S}}

\newcommand{\Real}{\mathbb{R}}

\newcommand{\norm}[1]{\Vert#1\Vert}

\def\<{\left\langle} \def\>{\right\rangle}
\def\({\left(} \def\){\right)}
\newcommand{\n}{\nabla}
\newcommand{\p}{\partial}



\begin{document}
\title[the v-Schr\"odinger flow]{Existence of weak solutions and regular solutions to the incompressible Schr\"odinger flow}
\thanks{*Corresponding Author}
\author{Bo Chen}
\address{School of Mathematics, South China University of Technology, Guangzhou, 510640, People's Republic of China}
\email{cbmath@scut.edu.cn}

\author{Guangwu Wang}
\address{School of Mathematics and Information Sciences, Guangzhou University, 510006, People's Republic of China}
\email{yunxianwgw@163.com}
	
\author{Youde Wang*}
\address{1. School of Mathematics and Information Sciences, Guangzhou University, 510006, People's Republic of China;
2. Hua Loo-Keng Key Laboratory of Mathematics, Institute of Mathematics, AMSS, and School of
Mathematical Sciences, UCAS, Beijing 100190, People's Republic of China.}
\email{wyd@math.ac.cn}

\begin{abstract}
In this paper, we are concerned with the initial-Neumann boundary value problem of the Schr\"{o}dinger flow for maps from a smooth bounded domain in an Euclidean space into $\mathbb{S}^2$. By adopting a novel method due to B. Chen and Y.D. Wang, we prove the existence of short-time regular solutions to this flow within the framework of Sobolev spaces when the underlying space is a smooth bounded domain in $\Real^m$ with $m\leq 3$. Moreover, we also utilize the ``complex structure approximation method" to establish the global existence of weak solutions to the incompressible Schr\"{o}dinger flow in a smooth bounded domain of $\mathbb{R}^m$ (where $m\geq 1$).
\end{abstract}

\maketitle
\section{Introduction}
The goals of this paper are to investigate the existence of weak solutions and regular solutions to the initial-Neumann boundary value problem of the incompressible Schr\"{o}dinger flow:
\begin{equation}\label{eq-ISMF}
	\begin{cases}
		\p_tu+\n_vu =u\times\De u,\quad\quad&\text{(x,t)}\in\Om\times \mathbb{ R}^+,\\[1ex]
		\frac{\p u}{\p \nu}=0, &\text{(x,t)}\in\p\Om\times \mathbb{ R}^+,\\[1ex]
		u(x,0)=u_0: \Om\to \mathbb{ S}^2,
	\end{cases}
\end{equation}
where $\Om\subset \mathbb{R}^m(m\geq 1)$ is a smooth bounded domain, $u$ is a time-dependent map from $\Om$ into a standard sphere $ \mathbb{S}^2$ and $\mbox{div}(v)=0$ inside $\Om$ for any $t\in \Real^+$. In some sense, the incompressible Schr\"{o}dinger flow can be viewed as a Schr\"{o}dinger flow from a underlying manifold with a time-dependent metric, and we will describe this in Subsection 1.1. 

So called Schr\"{o}dinger flow with variable metric is just a Schr\"{o}dinger flow from Riemannian manifold family $(M, g_t)$ into a K\"ahler manifold $(N, J)$ written by
$$\p_tu =J(u)\tau_{g_t} (u),$$
where $\tau_{g_t} (u)$ is the tension field of $u$ with respect to $g_t$. Indeed, Schr\"{o}dinger flow from a underlying manifold with a time-dependent metric into $\mathbb{S}^2$ appears as the Gauss map flow of a skew mean curvature flow (also referred to as binormal curvature flow), for details we refer to \cite{CS, CS1}.

To our best knowledge, there are few literatures on the wellposedness of the initial-Neumann boundary value problem of Schr\"{o}dinger flow with a variable metric family $g(t)$ denoted by $g_t$, where $t\in\mathbb{R}$. In fact, the existence of local strong solution or regular solutions to the initial-Neumann boundary value problem of the Schr\"{o}dinger flow from a Riemannian manifold $M$ with $\dim(M)\geq 4$ and fixed metric is still a long-standing open problem (see \cite{CW1}).

\subsection{Main model and Background}
Let $\Om$ be a bounded domain in $\Real^m$ with $m\leq 3$. For a time-dependent map $u$ from $\Om$ into $\U^2$, the well-known Landau-Lifshitz (LL) equation
\begin{equation}\label{eq-LL}
\p_t u=-u\times \De u	
\end{equation}
was initially proposed by Landau and Lifshitz\cite{LL} in 1935 as a phenomenological model for investigating the dispersive theory of magnetization in ferromagnets. Subsequently, in 1955, Gilbert\cite{G} introduced a modified version of the Landau-Lifshitz equation by incorporating with a dissipative term, which is now widely referred to as the Landau-Lifshitz-Gilbert equation. This equation is given by
\begin{equation*}
	\p_tu + \alpha u\times \Delta u + \beta u\times (u\times \Delta u)=0,
\end{equation*}
where $\beta$ is a real number and $\al \geq 0$ is called the Gilbert damping coefficient. Here ``$\times$" denotes the cross product in $\Real^{3}$ and $\De$ is the Laplace operator in $\Real^{3}$.

Let $v:\Om\times \Real^+\to \Real^m$ be a vector field, which satisfies $\mbox{div}(v)=0$ inside $\Om$.  For any constant $\ga\neq 0$, the following equation is called as the incompressible Schr\"odinger flow (or the incompressible LL equation):
\begin{equation}\label{eq-ILL}
\p_tu+\ga\n_vu =-u\times\De u.	
\end{equation}
This equation was derived by Chern et al \cite{Chern} as a model for the purely Eulerian simulation of incompressible fluids.

In the case of the vector field $v$ represents the velocity field in a magnetic fluid which satisfies a Navier-Stokes equation that includes a magnetic term, we can derive the so-called the Navier-Stokes-Schr\"odinger flow
\begin{equation}\label{eq-NSLL}
	\begin{cases}
		\p_t v+\n_v v+\n P=\mu\De v-\nabla\cdot(\nabla u\odot\nabla u),\\[1ex]
		\mbox{div}(v)=0,\\[1ex]
		\p_t u+\ga \n_v u=-u\times \De u.
	\end{cases}
\end{equation}
Here $\mu$ is a constant, $u: \Om^m\times \Real^+\to \U^2$ is the magnetization field, $v: \Om^m\times \Real^+\to \Real^m$ is the velocity field of the fluid and $P$ is the pressure function, where $\Om^m$ is a domain in $\Real^m$ with $m=2, 3$. The term $\n u\odot\n u$ is a $m\times m$ matrix with $(i,j)$-th entry
\[(\n u\odot\n u)_{ij}=\<\n_i u, \n_ju\>.\]		
This flow can be utilized to model the dispersive theory of magnetization in ferromagnets when one takes into account quantum effects.

If the vector field $v$ additionally satisfies $\ga\<v,\nu\>|_{\p \Om}=0$, where $\nu$ is the outward unit normal vector on the boundary $\p\Om$, it is worthy to point out that the incompressible LL equation \eqref{eq-ILL} is gauge equivalent to LL equation \eqref{eq-LL}. Indeed, let $\phi_t: \Om \to \Om$  be a family of diffeomorphisms of $\Om$ generated by $\ga v$, which preserves the volume element.  Namely, $\phi_t$ is the solution to the following ordinary differential equation (ODE)
\begin{equation}\label{ODE}
	\begin{cases}
		\frac{\p \phi}{\p t}=\ga v(\phi_t(x),t),\\[1ex]
		\phi(\cdot, 0)=\phi_0,
	\end{cases}
\end{equation}
where $\phi_0: \Om\to \Om$ is a given diffeomorphism. Let $u$ solve \eqref{eq-ILL}, and set $\tilde{u}(x,t)=u(\phi_t(x), t)$. Then we have
\[\p_t\tilde{u}=(\p_t u+\ga\n_v u)\circ\phi_t(x)=\phi_t^*(-u\times \De u)=-\tilde{u}\times \De_{g_t}\tilde{u},\]
where $\De_{g_t}$ is the Laplace operator induced by the the pull-back metric $g_t=\phi_t^*g$. This is the standard LL equation \eqref{eq-LL} with respect to the pull-back metric $g_t$. So, the incompressible LL equation \eqref{eq-ILL} can be regarded as a Schr\"odinger flow with time-dependent domain metric. 

\medskip
Now, let us review some relevant previous results in this field. In the last five decades, there has been significant advancement in the study of well-posedness for both weak and regular solutions of LL-type equations and the Schr\"odinger flow.

In 1985, Visintin \cite{V} established the existence of weak solutions to the LLG equation with magnetostrictive effects. Subsequently, in 1986, P. L. Sulem, C. Sulem, and C. Bardos \cite{SSB} utilized difference methods to prove the global existence of weak solutions and locally smooth solutions for the LL equation without a dissipation term (referred to as the Schr\"{o}dinger flow for maps into $ \mathbb{  S}^2$) defined on $ \mathbb{R}^n$. In 1992, Alouges and Soyeur \cite{AS} demonstrated a non-uniqueness result for weak solutions to the LLG equation with an initial-Neumann boundary condition, considering the unit ball $\Omega$ in $ \mathbb{R}^3$. In 1993, B.L. Guo and M.C. Hong \cite{GH} employed methods used for studying harmonic maps to establish the global existence and uniqueness of partially regular weak solutions for LLG equation. In 1998, Y.D. Wang \cite{W} demonstrated the existence of weak solutions to the Cauchy problem of the Schr\"{o}dinger flow (i.e. LL equation) for maps from an $n$-dimensional Euclidean domain $\Omega$ or a closed $n$-dimensional Riemannian manifold $M$ into a 2-dimensional unit sphere $ \mathbb{S}^2$, which largely improved the work \cite{SSB}.  Z.L. Jia and Y.D. Wang \cite{JW1,JW2} employed a method inspired by \cite{DWW,W} to achieve global weak solutions for a wide class of generalized Schr\"{o}dinger flows in a more general setting, where the base manifold is a bounded domain $ \mathbb{R}^n$ (where $n\geq 2$) or a compact Riemannian manifold $\mathbb{M}^n$, and the target space is $\mathbb{S}^2$ or the unit sphere $ \mathbb{S}_{\g}^n$ in a compact Lie algebra $\g$. Recently, B. Chen and Y.D. Wang \cite{CW0} improved the methods proposed by Wang \cite{W} to establish the global existence of weak solutions for the Landau-Lifshitz flows and heat flows associated with the micromagnetic functional, considering the initial-Neumann boundary condition.

The local existence and uniqueness of regular solutions or smooth solutions for the Schr\"odinger flow for maps from a closed Riemannian manifold or an Euclidean space into a complete K\"ahler manifold was demonstrated by W.Y. Ding and Y.D. Wang in \cite{DW,DW1}. For initial data with low regularity, the Schr\"{o}dinger flow from Euclidean space into a Riemann surface $X$ has been indirectly studied using the ``modified Schr\"{o}dinger map equations" and enhanced energy methods. For instance, A.R. Nahmod, A. Stefanov, and K. Uhlenbeck \cite{NSU} employed Picard iteration in suitable function spaces of the Schr\"{o}dinger equation to obtain a near-optimal (but conditional) local well-posedness result for the Schr\"{o}dinger map flow for maps from two dimensions into the standard sphere $X=\mathbb{S}^2$ or hyperbolic space $X=\mathbb{H}^2$. The resolution of the well-posedness hinges on the consideration of truly quatrilinear forms of weighted $L^2$-functions.

For the global existence in one dimension of the Schr\"{o}dinger flow from $\mathbb{S}^1$ or $\mathbb{R}^1$ into a K\"{a}hler manifold, references \cite{CSU,PWW,RRS,ZGT} and a recent preprint \cite{WZ} provide further details. The global well-posedness result for the Schr\"{o}dinger flow from $\mathbb{R}^n$ (where $n\geq3$) into $\mathbb{S}^2$ in critical Besov spaces was proven by Ionescu and Kenig in \cite{IK}, independently by Bejenaru in \cite{B1}, and later improved to global regularity for small data in critical Sobolev spaces for dimensions $n\geq 4$ in \cite{BIK}. The global well-posedness result for small data in critical Sobolev spaces in dimensions $n\geq 2$ was addressed in \cite{BIKT}. Recently, Z. Li in \cite{L1,L2} proved global results for the Schr\"{o}dinger flow from $\mathbb{R}^n$ (where $n\geq 2$) to compact K\"{a}hler manifolds with small initial data in critical Sobolev spaces.

F. Merle, P. Rapha\"{e}l, and I. Rodnianski \cite{MRR} investigated the energy critical Schr\"{o}dinger flow problem with a 2-sphere target for equivariant initial data of homotopy index $k = 1$. They established the existence of a codimension one set of well-localized smooth initial data arbitrarily close to the ground state harmonic map in the energy critical norm, leading to finite-time blowup solutions. They provided a sharp description of the corresponding singularity formation, which occurs through the concentration of a universal bubble of energy. Additionally, self-similar solutions to the Schr\"{o}dinger flow from $\mathbb{C}^n$ into $\mathbb{C}P^n$ with locally bounded energy that blow up at finite time were found in \cite{DTZ,GSZ}. Very recently, G.W. Wang and B.L. Guo \cite{WG} established a blowup criterion for the strong solution to the multi-dimensional Landau-Lifshitz-Gilbert equation.

Regarding traveling wave solutions with vortex structures, F. Lin and J. Wei \cite{LW} employed perturbation methods to consider such solutions for the Schr\"{o}dinger map flow equation with an easy-axis assumption. They demonstrated the existence of smooth traveling waves with bounded energy if the velocity of the traveling wave is sufficiently small. Moreover, they showed that the traveling wave solution possesses exactly two vortices. Later, J. Wei and J. Yang \cite{WY} considered the same Schr\"{o}dinger map flow equation as in \cite{LW}, which corresponds to the Landau-Lifshitz equation describing planar ferromagnets. They constructed a traveling wave solution with vortex helix structures for this equation and provided a complete characterization of the solution's asymptotic behavior using perturbation techniques.


On the other hand, the Landau-Lifshitz-Gilbert system with Neumann boundary conditions has garnered significant attention from both physicists and mathematicians. In 2001, Carbou and Fabrie established local existence of regular solutions for the LLG equation on bounded domains in $ \mathbb{R}^n$ (where $n\leq 3$) in \cite{CF}. Later, Carbou and Jizzini \cite{CJ} studied a model of ferromagnetic material subjected to an electric current and proved the local existence in time of very regular solutions for this model in Sobolev spaces. They also described in detail the compatibility conditions at the boundary for the initial data. Inspired by \cite{CJ}, B. Chen and Y.D. Wang \cite{CW,CW2}obtained the existence of locally very regular solution for LLG equation with spin-polarized transport, as well as for the Schr\"odinger flow with damping term for maps from a 3-dimensional manifold with boundary into a compact symplectic manifold, considering the Neumann boundary conditions.   Very recently, B. Chen and Y.D. Wang \cite{CW1,CW3} established the existence and uniqueness of local regular solutions (or local smooth solutions) for the challenging initial-Neumann boundary value problem of the Schr\"{o}dinger flow from a smooth bounded domain $\Om$ in $\mathbb{R}^3$ into $\mathbb{S}^2$:
\begin{equation}\label{eq-SMF}
\begin{cases}
\p_tu=u\times\De u,\quad\quad&\text{(x,t)}\in\Om\times \mathbb{ R}^+,\\[1ex]
\frac{\p u}{\p \nu}=0, &\text{(x,t)}\in\p\Om\times \mathbb{ R}^+,\\[1ex]
u(x,0)=u_0: \Om\to \mathbb{S}^2.
\end{cases}
\end{equation}

A natural question $Q_0$ arises: can we generalize our prior results in \cite{CW1,CW3} to address the initial-Neumann boundary value problem of the following Schr\"odinger flow governed by a time-dependent metric: 
	\[\p_t u=u\times \De_{g_t}u?\]
Here $u: (\Om, g_t)\to \mathbb{S}^2$, and $g_t$ is a variable metric family. This problem is intimately connected to the free boundary problem associated with skew mean curvature flow. However, tackling this problem necessitates navigating novel and inherent challenges stemming from the time-dependent metric $g_t$.

In the present paper, we provide a positive answer to problem $Q_0$ when $g_t$ exhibits self-similarity and is induced by a vector field $v: \Om\times\Real^+\to \Real^3$ satisfying the compatibility boundary condition $\<v,\nu\>|_{\p\Om}=0$. More precisely, $g_t=\phi_t^*g$ where $\phi_t$ solves \eqref{ODE} and $g$ is a fixed metric on $\Om$. Additionally, if $v$ satisfies the divergence-free condition, this special case of Schr\"odinger flow reduces to the incompressible Schr\"odinger flow.

By imposing appropriate regularity assumptions on the vector field $v$, we get the existence of global weak solutions and local regular solutions to the initial-Neumann boundary value problem to the incompressible Schr\"{o}dinger flow:
\begin{equation*}
\begin{cases}
\p_tu+\n_vu =u\times\De u,\quad\quad&\text{(x,t)}\in\Om\times \mathbb{ R}^+,\\[1ex]
\mbox{div}(v)=0, &\text{(x,t)}\in\Om\times \mathbb{R}^+,\\[1ex]
\frac{\p u}{\p \nu}=0, &\text{(x,t)}\in\p\Om\times \mathbb{R}^+,\\[1ex]
u(x,0)=u_0: \Om\to \mathbb{S}^2.
\end{cases}
\end{equation*}
 Our main results can be summarized as follows.
\subsection{Global weak solutions}
To state our first result on global well-posedness of the weak solutions to the incompressible Schr\"odinger flow \eqref{eq-ISMF}, we need to give the definitions of the weak solutions.

\begin{defn}[Weak solution]\label{def2}
Let $\Omega$ be a bounded smooth domain in $ \mathbb{R}^m$. Suppose that $v\in  L^2(\Real^+,L^\infty(\Omega))$, $\n v\in L^1(\Real^+, L^\infty(\Om))$, $u_0\in H^1(\Omega)$, $|u_0|=1$ a.e. in $\Omega$.  We say that $u\in L^\infty([0,T],H^1(\Omega))$ with $\p_tu\in L^{2}([0,T],H^{-1}(\Omega))$ is a weak solution to the incompressible Schr\"odinger flow \eqref{eq-ISMF} with initial data $u_0$ if $u$ satisfies that, for any $\varphi\in C^\infty(\bar{\Omega}\times[0,T])$,
\begin{align*}
   &\int_\Omega \left\langle u, \varphi\right\rangle dx(T)-\int_\Omega \left\langle u_0, \varphi \right\rangle dx(0)-\int_0^T\int_\Omega u\frac{\partial\varphi}{\partial t}+\int_0^T\int_\Omega \left\langle v\cdot \nabla u, \varphi\right\rangle  dxdt\\
    &+\int_0^T\int_\Omega \left\langle u\times \nabla u, \nabla \varphi\right\rangle dxdt=0,
\end{align*}
where $\bar{\Om}$ is the closure of $\Om$, and $u(x,t)\to u_0$ as $t \to 0$ in the space $C^0([0,T], L^2(\Om))$.
\end{defn}
\begin{thm}\label{th2}
Let $\Omega$ be a bounded smooth domain in $\mathbb{R}^m(m\geq 1)$. Suppose that $u_0\in H^1(\Omega, \mathbb{S}^2)$, $v\in  L^2(\Real^+,L^\infty(\Omega))$, $\n v\in L^1(\Real^+, L^\infty(\Om))$, $\textnormal{\mbox{div}}(v)=0$ for any $t\in\Real^+$ and $\<v,\nu\>|_{\p\Om\times \Real^+}=0$. Then, the incompressible Schr\"odinger flow \eqref{eq-ISMF} admits a global weak solution $u$ with initial data $u_0$ and $|u|=1$ for a.e. $(x,t)\in \Omega\times\Real^+$, which satisfies the following inequality
\begin{equation}\label{en-es}
\sup_{0\leq t\leq T} \|u\|_{H^1(\Omega)}^2
 \leq \exp\(2\int_0^T\|\nabla v\|_{L^\infty(\Omega)}(s)ds\)\int_\Omega |\nabla u_0|^2dx+\int_\Omega|u_0|^2 dx,
 \end{equation}
for any $0<T<\infty$.
\end{thm}

\begin{rem}
It is not difficult that we can also obtain the same results as in the above Theorem \ref{th2} if the domain $\Omega$ in Theorem \ref{th2} is replaced by a closed Riemannian manifold (for instance, a flat torus in $ \mathbb{  R}^n$).
\end{rem}

Theorem \ref{th2} is proved by using the complex structure approximation method originally from \cite{W}. Indeed,  for any $u\in \U^2$, $u\times:T_u\U^2\to T_u\U^2$ can be interpreted as a complex structure on $\U^2$, which rotates vectors in the tangent space of $\U^2$ by  $\frac{\pi}{2}$ degrees counterclockwise. This complex structure leads to  the following two important properties for equation \eqref{eq-ISMF}
\begin{itemize}
	\item[$(1)$] A priori estimate: If the initial data  $u_0\in H^1(\Omega, \mathbb{S}^2)$, $v\in L^1(\Real^+,W^{1,\infty}(\Omega))$, $\textnormal{\mbox{div}}(v)=0$ for any $t\in\Real^+$ and $\<v,\nu\>|_{\p\Om\times \Real^+}=0$, then the a priori estimate \eqref{en-es} holds true;
	\item[$(2)$] Divergence structure: The equation $u\times \De u=\mbox{div}(u\times \n u)$ holds, which reflects the divergence structure of the equation.
\end{itemize}
The above two properties play a crucial role on obtaining weak solutions to \eqref{eq-ISMF}. Hence, we consider the following approximation of the complex structure $u\times$:
\[J(u)=\frac{u}{\max\{|u|, 1\}},\]
and the corresponding approximation equation of \eqref{eq-ISMF}:
\begin{equation}\label{a-eq}
	\begin{cases}
		\p_t u+\n_vu=\ep \De u+J(u)\times \De u,\quad\quad&\text{(x,t)}\in\Om\times \mathbb{ R}^+,\\[1ex]
		\frac{\p u}{\p \nu}=0, &\text{(x,t)}\in\p\Om\times \mathbb{ R}^+,\\[1ex]
		u(x,0)=u_0: \Om\to \mathbb{ S}^2.
	\end{cases}
\end{equation}
It is noted that this equation exhibits a similar a priori estimate as mentioned in property (1). Moreover, once
we can show $|u|\leq 1$, this auxiliary equation also exhibits the same divergence structure as stated in property (2), namely
\[J(u)\times \De u=u\times \De u=\mbox{div}(u\times \n u).\]

Consequently, Theorem \ref{th2} can be established by demonstrating a uniform energy estimate (independent of $\ep$) for the approximation solution $u^\ep$ to \eqref{a-eq}, and taking a convergence argument to show that $u^\ep$ converges to a weak solution to \eqref{eq-ISMF} which satisfies the a priori estimate \eqref{en-es}.
\medskip
\subsection{Local regular solutions} Our second result is the existence of local regular solutions to \eqref{eq-ISMF}, which are the main conclusions of the present paper.
\begin{thm}\label{mth2}
Let $\Om$ be a smooth bounded domain in $\Real^m$ where $m\leq 3$. Let $u_0\in H^{3}(\Om)$ satisfy the compatibility condition:
	\[\frac{\p u_0}{\p \nu}|_{\p \Om}=0.\]
Suppose that $v\in L^\infty(\Real^+,W^{1,3}(\Om))\cap C^0(\Real^+, H^1(\Om))\cap L^4(\Real^+, L^\infty(\Om))$, $\n v\in L^2(\Real^+, L^\infty(\Om))$, $\p_t v\in L^2(\Real^+, H^1(\Om))$, $\textnormal{\mbox{div}}(v)=0$ inside $\Om$ for any $t\in\Real^+$ and $\<v,\nu\>|_{\p\Om\times \Real^+}=0$. Then there exists constants $T_0$  and $C(T_0)$ depending only on $\norm{u_0}_{H^3}$, $\norm{v}_{L^\infty(\Real^+,W^{1,3})}$ and the $L^1$-norm of $f(t)=\norm{\p_tv}^2_{H^1}+\norm{v}^4_{L^\infty}+\norm{\n v}^2_{L^\infty}$, such that the problem \eqref{eq-ISMF} admits a local solution $u\in L^\infty([0,T_0],H^3(\Om, \U^2))$, which satisfies
\begin{equation}\label{es-mth1}
\sup_{0\leq t\leq T_0}\(\norm{u}^2_{H^3(\Om)}+\norm{\p_tu}^2_{H^1(\Om)}\)\leq C(T_0).
\end{equation}
\end{thm}

%
%

We will only show Theorem \ref{mth2} for the case when the dimension of $\Om$ is 3, as the lower dimensional cases can be demonstrated in a similar manner. The proof of Theorem \ref{mth2} follows a similar argument with that presented in \cite{CW1}, but we need to overcome some new difficulties originated from the vector field $v$. We utilize the local regular solution $u_\ep$ to the following intrinsic parabolic approximation equation for \eqref{eq-ISMF}:
\begin{equation}\label{eq-ASMF}
	\begin{cases}
		\p_tu=\ep\tau_v(u)+u\times\tau_v(u),\quad\quad&\text{(x,t)}\in\Om\times \Real^+,\\[1ex]
		\frac{\p u}{\p \nu}=0, &\text{(x,t)}\in\p\Om\times \Real^+,\\[1ex]
		u(x,0)=u_0: \Om\to \U^2,
	\end{cases}
\end{equation}
that has been given by Carbou and Jizzini in \cite{CJ} (or to see Theorem \ref{mth1} for the details) to approximate a regular solution to \eqref{eq-ISMF}. For simplicity, we usually set \[\tau_v(u)=\tau(u)+u\times\n_v u=\De u+|\n u|^2u+u\times\n_vu.\]

The approximate equation \eqref{eq-ASMF} preserves the inherent geometric structures of the incompressible Schr\"odinger flow:
\begin{itemize}
	\item[$(a)$] For any point $(x,t)$, the equation $\p_tu=\ep\tau_v(u)+u\times\tau_v(u)$ resides within the tangent space $T_{u(x,t)}\mathbb{S}^2$ of the sphere $\mathbb{S}^2$ at the point $u(x,t)$. This ensures that the solution $u_\ep$ to \eqref{eq-ASMF} remains confined to the surface of $\mathbb{S}^2$. Consequently, we can apply the geometric properties of $\mathbb{S}^2$ to derive more precise energy estimates for $u_\ep$;
	\item[$(b)$] The two terms on the right hand of approximate equation \eqref{eq-ASMF} are orthogonal to each other, which implies that $\p_t u_\ep$, $\De \p_tu_\ep$ and $\De \tau_v(u_\ep)$ are suitable test function that comply with the Nuemann boundary conditions when establishing energy estimates for $u_\ep$.
\end{itemize}

The crux of this proof lies on demonstrating a uniform $H^3$-estimate of approximate solution $u_\ep$ with respect to $\ep\in (0,1)$. To achieve this, we establish a critical equivalent norm estimate of $\norm{u_\ep}_{H^3}$, which is given by
\[\norm{u_\ep}^2_{H^3}\leq C(1+\norm{u_\ep}^2_{H^2}+\norm{\frac{\p u_\ep}{\p t}}^2_{H^1}+\norm{v}^2_{W^{1,3}})^3,\]
where $C$ is a constant independent of $u_\ep$. When $v\in L^\infty(\Real^+, W^{1,3}(\Om))$, this estimate implies that obtaining a uniform estimate of $\norm{u_\ep}_{H^3}$ is equivalent to acquiring a uniform bound for the auxiliary functional:
\[G(u_\ep)=\norm{u_\ep}^2_{H^2}+\norm{\frac{\p u_\ep}{\p t}}^2_{H^1}.\]

The above estimate strongly suggests that we should focus on studying the equation for $\p_tu_\ep$. By utilizing the properties of cross product ''$\times$" on $\Real^3$ and the complex structure $u\times: T_u\U^2\to T_u\U^2$ respectively, we can derive the following fine form for the equation of $\p_t u_\ep$:
\begin{equation}\label{K}
	\begin{aligned}
		&\p_t\p_tu_\ep+(1-\ep^2)\De \tau_v(u_\ep)-2\ep\De(u_\ep\times \tau_v(u_\ep))\\
		=&-\ep\{2\n u_\ep\dot{\times}\n\tau_v(u_\ep)+\De u_\ep\times(|\n u_\ep|^2u_\ep+u_\ep\times\n_vu_\ep)\}\\
		&+\ep\{u_\ep\times\n_v\p_tu_\ep+|\n u_\ep|^2\p_tu_\ep+u_\ep\times\n_{\p_tv}u_\ep\}\\
		&+|\n u_\ep|^2\tau_v(u_\ep)-2\<\n u_\ep,\tau_v(u_\ep)\>\cdot\n u_\ep-\n_v\p_tu_\ep\\
		&-\n_{\p_tv}u_\ep+f_1u_\ep+\p_tu_\ep\times f_2,	
	\end{aligned}
\end{equation}
where $\n u_\ep\dot{\times}\n\tau_v(u_\ep)=\sum_{i=1}^m\n_iu_\ep\otimes\n_i\tau_v(u_\ep)$, and
\begin{align*}
	f_1=&\<\De\tau_v(u_\ep), u_\ep\>-\<\p_tu_\ep,\n_vu_\ep\>+2\ep\<\n \p_t u_\ep, \n u_\ep\>,\\
	f_2=&\De u_\ep+u_\ep\times\n_vu_\ep+\ep\n_v u_\ep.
\end{align*}

Subsequently, by employing the geometric structure $(a)$ and $(b)$, we discover that $\p_tu_\ep$ and $\De\p_t u_\ep$ are appropriate test functions that align with equation \eqref{K} since $\frac{\p \p_t u_\ep}{\p \nu}|_{\p\Om}=0$ for all $t\geq 0$. Selecting these two test functions for \eqref{K} allows us to obtain the desired estimate of $G(u_\ep)$. This process involves a meticulous utilization of the geometric information inherent in the target manifold $(\U^2, J=u\times)$ as mentioned in the authors' previous work \cite{CW1}. Additionally, we also capitalize on the assumption of $v$:
\[\textnormal{\mbox{div}}(v)=0 \,\,\text{in} \,\,\Om\times \Real^+ \quad \text{and}\quad  \<v,\nu\>|_{\p\Om\times \Real^+}=0.\]
For instance, when selecting $\p_t u_\ep$ as test function for \eqref{K}, we can use the properties $|u_\ep|\equiv 0$ and $\<c\times d, c\>=0$ for any vectors $c,d$, to demonstrate that
\[\int_{\Om}f_1\<u_\ep,\p_t u_\ep\>dx=0 \,\,\, \text{and}\,\,\, \int_{\Om}\<\p_tu_\ep\times f_2,\p_t u_\ep\>dx=0,\]
despite the complicity of the terms $f_1$ and $f_2$. Furthermore, by applying the assumptions made about $v$, we can also show that the term$\int_{\Om}\<\n_v\p_t u_\ep, \p_t u_\ep\>dx=0$. For the comprehensive uniform $H^3$-estimate of $u_\ep$, one can refer to Section \ref{s: reg-solu} for the details. 

\begin{rem}\
	
\begin{itemize}

\item[$(1)$] Since our proofs for Theorem \ref{th2} and Theorem \ref{mth2} rely heavily on the assumption that $v$ is divergence-free, it seems that our current arguments may not be valid when $v$ is not divergence free. This naturally leads to the question $Q_1$: Can we show the existence of global weak solutions or local regular solutions to the problem \eqref{eq-ISMF} where $v$ is a general time dependent vector field? This is a challenging problem that requires further investigation and possibly new techniques.

\item[$(2)$] Motivated by our previous result \cite{CW3}, we pose the question $Q_2$: What compatibility boundary conditions on $u_0$ and $v$ can guarantee the existence of very regular solutions to the initial-Neumann boundary value problem to the incompressible Schr\"{o}dinger flow. This question explores the role of boundary conditions in determining the regularity of solutions and is an important direction for our future research.
\end{itemize}	
\end{rem}

\medskip
The rest part of this paper will be organized as follows. In Section \ref{s: pre}, we provide the necessary background on Sobolev spaces and present preliminary lemmas. In Section \ref{s: w-sol}, we establish the global existence of the incompressible Schr\"{o}dinger flow with Neumann boundary conditions in a bounded domain in $ \mathbb{R}^m$ with $m\geq 1$. Finally, Section \ref{s: reg-solu} is dedicated to proving the existence of local regular solutions for the incompressible Schr\"{o}dinger flow.

\section{Preliminary}\label{s: pre}
\subsection{Notations}

In this section, we start with recalling some notations on Sobolev spaces which will be used in following context. Let $\Om$ be a smooth bounded domain in $\Real^n$ with $n\in \mathbb{N}$, $u=(u_1,u_2,u_3):\Om\to\U^{2}\hookrightarrow\Real^3$ be a map. We set
\[H^{k}(\Om,\U^{2})=\{u\in H^{k}(\Om):|u|=1\,\,\text{for a.e. x}\in \Om\},\]
where we denote $H^k(\Om)=W^{k,2}(\Om, \Real^3)$.

Moreover, let $(B,\norm{.}_B)$ be a Banach space and $f:[0,T]\to B$ be a map. For any $p>0$ and $T>0$, recall that
\[\norm{f}_{L^p([0,T], B)}:=\(\int_{0}^{T}\norm{f}^p_{B}dt\)^{\frac{1}{p}},\]
and
\[L^p([0,T],B):=\{f:[0,T]\to B:\norm{f}_{L^p([0,T],B)}<\infty\}.\]
In particular, we denote
\[L^{p}([0,T],H^{k}(\Om,\U^{2}))=\{u\in L^{p}([0,T],H^{k}(\Om,\Real^{3})):|u|=1\,\,\text{for a.e. (x,t)}\in \Om\times[0,T]\},\]
where $k,\,l\in \mathbb{N}$  and $p\geq 1$.

Without lose of generality and for simplicity, we always use $C$ to denote constants independent of $\ep$ appearing in energy estimates in the subsequent context.

\medskip
\subsection{Preliminary lemmas}
Next, for later application, we need to recall some critical lemmas.
\begin{lem}\label{eq-norm}
	Let $\Om$ be a bounded smooth domain in $\Real^{m}$ and $k\in \mathbb{  N}$. There exists a constant $C_{k,m}$ such that, for all $u\in H^{k+2}(\Om)$ with $\frac{\p u}{\p \nu}|_{\p\Om}=0$,
	\begin{equation}\label{eq-n}
		\norm{u}_{H^{2+k}(\Om)}\leq C_{k,m}(\norm{u}_{L^{2}(\Om)}+\norm{\De u}_{H^{k}(\Om)}).
	\end{equation}
	Here, for simplicity we denote $H^0(\Om):=L^2(\Om)$.
\end{lem}
In particular, the above lemma implies that we can define the $H^{k+2}$-norm of $u$ as follows
\[\norm{u}_{H^{k+2}(\Om)}:=\norm{u}_{L^2(\Om)}+\norm{\De u}_{H^k(\Om)}.\]

\begin{lem}\label{Gron-inq}
Let $f: \Real^+\to \Real^+$ be a nondecreasing continuous function such that $f>0$ on $(0,\infty)$ and $\int_{1}^{\infty}\frac{1}{f}dx<\infty$. Let $y$ be a continuous function which is nonnegative on $\Real^+$ and let $g$ be a nonnegative function in $L^{1}_{loc}(\Real^+)$. We assume that there exists a $y_0>0$ such that for all $t\geq0$, we have the inequality
\[y(t)\leq y_0+\int_{0}^{t}g(s)ds+\int_{0}^{t}f(y(s))ds.\]
Then, there exists a positive number $T^*$ depending only on $y_0$, $g$ and $f$, such that for all $T<T^*$, there holds
\[\sup_{0\leq t\leq T}y(t)\leq C(T,y_0),\]
for some constant $C(T,y_0)$.
\end{lem}

\begin{lem}[Theorem II.5.16 in \cite{BF} or \cite{Sim}]\label{A-S}
Let $X\subset B\subset Y$ be Banach spaces. Suppose that the embedding $B\hookrightarrow Y$ is continuous and that the embedding $X\hookrightarrow B$ is compact. Let $1\leq p,q,r\leq \infty$. For $T>0$, we define
\[E_{p,r}=\{f\in L^{p}((0,T), X), \frac{d f}{dt}\in L^{r}((0,T), Y)\},\]
which equipped a norm $\norm{f}:=\norm{f}_{L^{p}((0,T), X)}+ \norm{\frac{d f}{dt}}_{L^{r}((0,T), Y)}$.
Then, the following properties hold true.
	\begin{itemize}
		\item[$(1)$] If $p< \infty$, then the embedding $E_{p,r}$ in $L^p((0,T), B)$ is compact.
		\item[$(2)$] If $p< \infty$ and $p<q$, the embedding $E_{p,r}\cap L^q((0,T), B)$ in $L^s((0,T), B)$ is compact for all $1\leq s<q$.
		\item[$(3)$] If $p=\infty$ and $r>1$, the embedding of $E_{p,r}$ in $C^0([0,T], B)$ is compact.
	\end{itemize}
\end{lem}
\begin{lem}[Theorem II.5.14 in \cite{BF}]\label{C^0-em}
Let $k\in  \mathbb{  N}$, then the space
	\[E_{2,2}=\{f\in L^{2}((0,T),H^{k+2}(\Om)),\frac{\p f}{\p t}\in L^{2}((0,T), H^k(\Om))\}\]
is continuously embedded in $C^0([0,T], H^{k+1}(\Om))$.
\end{lem}

\section{Global weak solutions}\label{s: w-sol}

In this section, we prove the global existence of the weak solution to the incompressible Sch\"odinger flow \eqref{eq-ISMF}. For this end we adopt the following approximate equation
\begin{equation}\label{app-ISMF}
	\begin{cases}
		\p_tu+\n_v u=\varepsilon\Delta u+J(u)\times\De u\quad &\mbox{in}\,\, \Omega\times \Real^+,\\[1ex]
		u(0,\cdot)=u_0:\,\Omega\rightarrow \U^2,\quad\quad \frac{\partial u}{\partial\nu}=0\quad &\mbox{on}\,\, \partial\Om\times \Real^+,
	\end{cases}
\end{equation}
where $0<\varepsilon<1$ is a positive constant, $J(u)$ is defined by
\[J(u)\equiv\frac{u}{\max\{1,|u|\}}.\]
Here the vector field $v$ satisfies that $v\in  L^2(\Real^+,L^\infty(\Omega))$, $\n v\in L^1(\Real^+, L^\infty(\Om))$, $\textnormal{\mbox{div}}(v)=0$ inside $\Om$ for any $t\in\Real^+$ and $\<v,\nu\>|_{\p\Om\times \Real^+}=0$. 

Next, we will construct a weak solution of \eqref{app-ISMF} by the classical Galerkin Approximation method and then show some a priori estimates on its solutions.
\subsection{Galerkin approximation and a priori estimates}
Let $\Omega$ be a bounded smooth domain in $ \mathbb{R}^m$, $\lambda_i$ be the $i$-th eigenvalue of the operator $\Delta-I$ with Neumann boundary condition, whose corresponding eigenfunction is $f_i$, that is
\begin{equation*}
	\begin{cases}
		(\Delta-I)f_i=-\lambda_i f_i &x\in \Om,\\[1ex]
		\frac{\partial f_i}{\partial \nu}|_{\partial \Omega}=0.
	\end{cases}
\end{equation*}

Without loss of generality, we assume $\{f_i\}_{i=1}^\infty$ are completely standard orthogonal basis of $L^2(\Omega, \mathbb{  R}^n)$. Let $H_n=span\{f_1,f_2,\cdots,f_n\}$ be a finite subspace of $L^2$, $P_n:L^2\rightarrow H_n$ be the Galerkin projection. In fact, for any $f\in L^2$, $f_n=P_n f=\sum_{i=1}^n\left\langle f,f_i\right\rangle _{L^2}f_i$,
and $\lim_{n\rightarrow\infty}\|f-f_n\|_{L^2}=0$.

Inspired by \cite{W, CW0}, we can choose the following Galerkin approximation equation associated with \eqref{app-ISMF}
\begin{equation}\label{Ga-app}
\begin{cases}
	\partial_t u_n^\varepsilon-\varepsilon\Delta u_n^\varepsilon=P_n\{-v\cdot \nabla u_n^\varepsilon+ J(u_n^\varepsilon)\times \Delta u_n^\varepsilon\}, &(x,t)\in \Omega\times \Real^+,\\[1ex]
		u_n^\varepsilon(x,0)=P_n(u_0)(x).
\end{cases}
\end{equation}

Let $u_n^\varepsilon(x,t)=\sum_{i=1}^n g_i^n(t)f_i(x)$, $g^n(t)=(g_1^n(t),g_2^n(t),\cdots,g_n^n(t))$ be a vector value function. Then, by a direct calculation, we have that $g^n(t)$ satisfies the following ordinary differential equation (ODE)
\begin{equation*}
\begin{cases}
\frac{\partial g^n}{\partial t}=F(t,g^n(t)), \\[1ex]
g^n(0)=(\left\langle u_0,f_1\right\rangle ,\cdots,\left\langle u_0,f_n\right\rangle ),
\end{cases}
\end{equation*}
where $F(g^n)$ is locally Lipshitz on $g^n$, since $J(y)$ is locally Lipshitz on $y$. Hence, there exists a solution $u_n^\varepsilon$ to the problem \eqref{Ga-app} on $[0,T_n)$, where $T_n>0$ is the maximal existence time for the above ODE.

Afterwards, we show uniform energy estimates for the approximation solution $u^\ep_n$ with respect to $n$.

\begin{lem}\label{lm5}
Assume $u_0\in L^2(\Omega)$, then there holds that
\begin{equation}\label{eq4.1}
\sup_{0\leq t\leq T}\int_\Omega |u_n^\varepsilon|^2 dx+2\varepsilon\int_0^T\int_\Omega |\nabla u_n^\varepsilon|^2dxdt \leq \int_\Omega |u_0|^2dx,
\end{equation}
for any $0<T<T_n$. Moreover, this estimate implies that $T_n=+\infty$.
\end{lem}
\begin{proof}
Multiplying the equation \eqref{Ga-app} by $u_n^\varepsilon$, and integrating by parts, we have
\begin{align*}
& \int_\Omega \partial_t u_n^\varepsilon \cdot u_n^\varepsilon dx =\frac{1}{2}\partial_t\int_\Omega|u_n^\varepsilon|^2dx,\\
& \int_\Omega v\cdot \nabla u_n^\varepsilon\cdot u_n^\varepsilon dx=\int_{\partial\Omega} v\cdot\nu \frac{1}{2}|u_n^\varepsilon|^2ds-\int_\Omega\mbox{div}v \frac{1}{2}|u_n^\varepsilon|^2dx=0,\\
&\varepsilon\int_\Omega \Delta u_n^\varepsilon \cdot u_n^\varepsilon dx=\int_{\partial\Omega} u_n^\varepsilon\cdot\frac{\partial u_n^\varepsilon}{\partial\nu}ds-\varepsilon\int_\Omega |\nabla u_n^\varepsilon|^2dx=-\varepsilon\int_\Omega |\nabla u_n^\varepsilon|^2dx,\\
& \int_\Omega (J(u_n^\varepsilon)\times \Delta u_n^\varepsilon)\cdot u_n^\varepsilon dx= \int_{\Omega} (\frac{u_n^\varepsilon}{\max\{|u^\ep_n,1|\}}\times \Delta u_n^\varepsilon) \cdot u_n^\varepsilon dx=0.
\end{align*}
Then we can easily derive the desired inequality \eqref{eq4.1} from the above estimates. 
	
On the other hand, since $\<g^n(t),g^n(t)\>=\norm{u^\ep_n}^2_{L^2(\Om)}(t)$ for any $t<T_n$, the above estimate for $u^\ep_{n}$ tells us that
\[\sup_{0<t<T_n}|g^n(t)|^2\leq \norm{u_0}^2_{L^2(\Om)}.\]
This implies that $T_n=\infty$.	
\end{proof}

\begin{lem}\label{lm6}
If $u_0\in H^1(\Omega)$ and $\n v\in L^1([0,T],L^\infty(\Omega))$, there holds that
\begin{equation}\label{eq4.2}
\begin{aligned}
&\sup_{0\leq t\leq T}\int_\Omega|\n u_n^\varepsilon|^2dx \leq  \exp(I(T)) \int_\Omega |\nabla u_0|^2dx,\\
&2\varepsilon\int_0^T\int_\Omega|\Delta u_n^\varepsilon|^2dxdt \leq (1+I(T)\exp(I(T))) \int_\Omega |\nabla u_0|^2dx,
\end{aligned}
\end{equation}
for any $0<T<\infty$, where $I(T)=2\int_0^T\|\nabla v\|_{L^\infty(\Omega)}(s)ds$.
\end{lem}

\begin{proof}
Multiplying the equation \eqref{Ga-app} by $-\Delta u_n^\varepsilon$, and integrating by parts, we have
\begin{align*}
& \int_\Omega \partial_t u_n^\varepsilon \cdot \Delta u_n^\varepsilon dx
=\int_{\partial\Omega} \partial_t u_n^\varepsilon\cdot \frac{\partial u_n^\varepsilon}{\partial\nu}ds- \frac{1}{2}\partial_t \int_\Omega |\nabla u_n^\varepsilon|^2dx
=- \frac{1}{2}\partial_t \int_\Omega |\nabla u_n^\varepsilon|^2dx,\\
&    \int_\Omega v\cdot\nabla u_n^\varepsilon \cdot \Delta u_n^\varepsilon dx\\
=&\int_{\partial\Omega} v\cdot\nabla u_n^\varepsilon \frac{\partial u_n^\varepsilon}{\partial \nu}ds-\sum_{i,j,k=1}^n\int_\Omega \partial_k v_i\partial_i (u_n^\varepsilon)_j\partial_k (u_n^\varepsilon)_jdx-\int_{\partial\Omega} v\cdot\nu \frac{1}{2}|\nabla u_n^\varepsilon|^2ds\\
=&-\sum_{i,j,k=1}^n\int_\Omega \partial_k v_j\partial_i (u_n^\varepsilon)_j\partial_k (u_n^\varepsilon)_jdx
\leq   \|\nabla v\|_{L^\infty(\Omega)}\int_\Omega|\nabla u_n^\varepsilon|^2dx,\\
& \int_\Omega\varepsilon\Delta u_n^\varepsilon \cdot \Delta u_n^\varepsilon dx=\varepsilon \int_\Omega |\Delta u_n^\varepsilon|^2dx,\\
&\int_\Omega J(u_n^\varepsilon)\times \Delta u_n^\varepsilon \cdot \Delta u_n^\varepsilon dx=0.
\end{align*}
Then we have
\begin{equation*}
\partial_t  \int_\Omega |\nabla u_n^\varepsilon|^2dx+2\varepsilon \int_\Omega |\Delta u_n^\varepsilon|^2dx\leq 2\|\nabla v\|_{L^\infty(\Omega)}\int_\Omega|\nabla u_n^\varepsilon|^2dx.
\end{equation*}
Using the Gronwall inequality and that fact that
\[\int_{\Om}|\n (P_n(u_0))|^2dx\leq \int_{\Om}|\n u_0|^2dx,\]
for any $T<\infty$ we can obtain
\begin{align*}
&\sup_{0\leq t\leq T}\int_\Omega|\nabla u_n^\varepsilon|^2 dx\leq \exp(I(T)) \int_\Omega |\nabla u_0|^2dx,\\
&2\varepsilon\int_0^T\int_\Omega|\Delta u_n^\varepsilon|^2dxdt \leq (1+I(T)\exp(I(T))) \int_\Omega |\nabla u_0|^2dx,
\end{align*}
where $I(T)=2\int_0^T\|\nabla v\|_{L^\infty(\Omega)}(s)ds$.
\end{proof}

From Lemma \ref{lm5} and Lemma \ref{lm6}, we can get the following estimate for $\partial_t u_n^\varepsilon$.
\begin{lem}\label{lm7}
Assume $u_0\in H^1(\Omega)$, $v\in L^2([0,T],L^\infty(\Omega))$ and $\n v\in L^1([0,T], L^\infty(\Om))$, there holds that
	
\begin{equation}\label{eq4.3}
\ep\int_{\Om}|\p_t u^\ep_n|^2dx\leq (1+\ep+2\ep S(T))(1+I(T)\exp(I(T))) \int_\Omega |\nabla u_0|^2dx,
\end{equation}
where $S(T)=\int_{0}^T\norm{v}^2_{L^\infty}ds$.
\end{lem}
\begin{proof}
From the equation \eqref{Ga-app}, we apply a simple computation to show
\begin{equation*}
\begin{split}
\int_0^T\int_{\Om}|\p_t u^\ep_n|^2dxdt\leq& 2(1+\ep) \int_0^T\int_{\Om}|\De u^\ep_n|^2dxdt+2 \int_0^T\int_{\Om}|v|^2|\n u^\ep_n|^2dxdt\\
\leq &\frac{1+\ep}{\ep}(1+I(T)\exp(I(T))) \int_\Omega |\nabla u_0|^2dx+2S(T)\sup_{0<t<T}\int_{\Om}|\n u^\ep|^2dx\\
\leq &(\frac{1+\ep}{\ep}+2S(T))(1+I(T)\exp(I(T))) \int_\Omega |\nabla u_0|^2dx.
\end{split}
\end{equation*}
Here have used the estimate \eqref{eq4.2}, and have denoted $S(T)=\int_{0}^T\norm{v}^2_{L^\infty}ds$.
\end{proof}
Therefore, we can get the following estimates.
\begin{prop}\label{prop1}
Suppose that $u_0\in H^1(\Omega)$, $v\in L^2(\Real^+,L^\infty)$, $\n v\in L^1(\Real^+, L^\infty)$, $\textnormal{\mbox{div}}(v)=0$ inside $\Om$ for any $t\in\Real^+$ and $\<v,\nu\>|_{\p\Om\times \Real^+}=0$. For any $n\in \mathbb{N}$ and $T> 0$, there exists a solution $u_n^\varepsilon \in L^\infty([0,T],H^1(\Omega))\cap L^2([0,T],H^2(\Omega))$ and $\partial_t u_n^\varepsilon\in L^2([0,T],L^2(\Omega))$ to \eqref{Ga-app}. Moreover,  the solution $u^\ep_n$ satisfies the a priori estimates \eqref{eq4.1}, \eqref{eq4.2} and \eqref{eq4.3}.
\end{prop}

Next, we will consider the compactness of the approximation solution $u_n^\varepsilon$. The main tool is well known Alaoglu's theorem and the Aubin-Lions-Simon compact lemma \ref{A-S}. Thus from the Proposition \ref{prop1} we know there exists a subsequence of $\{u_n^\varepsilon\}$, we still denote it by $\{u_n^\varepsilon\}$, and a $u_n^\varepsilon \in L^\infty([0,T],H^1(\Omega))\cap L^2([0,T],H^2(\Omega))$ and $\partial_t u_n^\varepsilon\in L^2([0,T],H^{-1}(\Omega))$, such that
\begin{align}
	& u_n^\varepsilon\rightharpoonup u^\varepsilon, ~weakly\ast ~in~ L^\infty([0,T],H^1(\Omega)), \label{eq4.6}\\
	& u_n^\varepsilon\rightharpoonup u^\varepsilon, ~weakly~in ~L^2([0,T],H^2(\Omega)).\label{eq4.7}
\end{align}
Next,  let $X=H^2(\Omega)$, $B=H^1(\Omega)$ and $Y=L^2(\Omega)$, then Lemma \ref{A-S} implies that
\begin{equation}\label{eq4.8}
	u_n^\varepsilon \rightarrow u^\varepsilon,~strongly~in~L^p([0,T],H^1(\Omega)),
\end{equation}
for any $p<\infty$.

\begin{thm}\label{lm9}
The limit $u^\ep$ of the sequence $\{u_n^\varepsilon\}$ is a strong solution of the problem \eqref{app-ISMF}, which satisfies the same estimates as those for $u^\ep_n$ in \eqref{eq4.2} and \eqref{eq4.2}. 
\end{thm}
\begin{proof}
	
For any $\phi\in C^\infty(\bar{\Om}\times [0,T])$, the approximation solution $u^\ep_n$ satisfies
\begin{equation*}
\begin{split}
&\int_0^T\int_\Omega \<\partial_t u^\varepsilon_n,\varphi\>dxdt+ \int_0^T\int_\Omega \<P_n(v\cdot \nabla u^\varepsilon_n), \varphi\>dxdt\\
=&\int_0^T\int_\Omega\<P_n(J(u^\varepsilon_n)\times \Delta u^\varepsilon_n), \varphi\> dxdt +\varepsilon\int_0^T\int_\Omega \<\Delta u^\varepsilon_n, \varphi\>dxdt.
\end{split}
\end{equation*}
	
From the equation \eqref{eq4.6}-\eqref{eq4.8}, we can derive that
\begin{align*}
& \partial_t u_n^\varepsilon \rightarrow \partial_t u^\varepsilon ~~~ weakly* ~in ~L^1([0,T],L^2(\Omega)), \\
& \Delta u_n^\varepsilon \rightarrow \Delta u^\varepsilon ~~~weakly ~in ~L^2([0,T],L^2(\Omega)),\\
& u_n^\varepsilon\rightarrow u^\varepsilon ~~~ strongly ~in ~C^0([0,T],H^1(\Omega)),\\
& u_n^\varepsilon \rightarrow u^\varepsilon ~~~a.e.~~(x,t)\in\Omega\times[0,T]\\
& J(u_n^\varepsilon)\rightarrow J(u^\varepsilon)~~~a.e. ~~(x,t)\in \Omega\times [0,T].
\end{align*}
These convergence results implies that
\begin{align*}
& \int_0^T\int_\Omega \<\partial_t u_n^\varepsilon,\varphi\>dxdt\rightarrow \int_0^T\int_\Omega \<\partial_t u^\varepsilon,\varphi\>dxdt, \\
& \int_{\Om}\<\De u^\ep_n, \varphi\>dx\rightarrow\int_{\Om}\<\De u^\ep, \varphi\>dx,\\
& \int_0^T\int_\Omega\<P_n(J(u_n^\varepsilon)\times \Delta u_n^\varepsilon), \varphi\> dxdt\rightarrow \int_0^T\int_\Omega\<(J(u^\varepsilon)\times \Delta u^\varepsilon),\varphi\>dxdt.
\end{align*}

Therefore, to prove $u^\ep$ is a strong solution to \eqref{app-ISMF}, we still need to show the convergence for that term $\int_0^T\int_\Omega \<v\cdot \nabla u_n^\varepsilon,\varphi\>dxdt$. By applying the fact $\mbox{div}(v)=0$ inside $\Om$ and $\<v, \nu\>_{\p\Om\times \Real^+}$, we have
\begin{align*}
\int_0^T\int_\Omega \<P_n(v\cdot \nabla u_n^\varepsilon), \varphi\> dxdt=&-\int_0^T\int_\Omega \<v\cdot \nabla P_n(\varphi), u^\ep_n\>dxdt\\
\rightarrow&-\int_0^T\int_\Omega \<v\cdot \nabla \varphi, u^\ep\>dxdt\\
=&\int_0^T\int_\Omega \<v\cdot \nabla u^\ep, \varphi\> dxdt.
\end{align*}

It remains that we need to check the Neumann boundary condition. Since for any $\psi \in C^\infty(\bar{\Omega}\times[0,T])$, there holds
\begin{equation*}
\int_0^T\int_\Omega \<\Delta u_n^\varepsilon, \psi\>dxdt=-\int_0^T\int_\Omega \<\nabla u_n^\varepsilon, \nabla \psi\>dxdt.
\end{equation*}
Let $n\rightarrow +\infty$, we have
\begin{equation*}
\int_0^T\int_\Omega \<\Delta u^\varepsilon,\psi\>dxdt=-\int_0^T\int_\Omega\<\nabla u^\varepsilon, \nabla \psi\>dxdt.
\end{equation*}
that is $\frac{\partial u^\varepsilon}{\partial \nu}|_{\partial\Omega \times [0,T]}=0$.
\end{proof}

To proceed, we need to show the following maximal principle for equation \eqref{app-ISMF}. 
\begin{lem}\label{lm10}
Let $u^\varepsilon$ be the solution that we have obtained in Theorem \ref{lm9}, which solves the following equation
\begin{equation}\label{eq-app}
\begin{cases}
\partial_t u^\varepsilon+ v\cdot \nabla u^\varepsilon=\varepsilon\Delta u^\varepsilon+J(u^\varepsilon)\times \Delta u^\varepsilon, &(x,t)\in \Om\times(0,T)\\[1ex]
\frac{\partial u^\varepsilon}{\partial \nu}|_{\partial \Omega\times (0,T)}=0,\\[1ex]
u^\varepsilon(x,0)=u_0:\Om\to \U^2.
\end{cases}
\end{equation}
	
Then $|u^\varepsilon|\leq 1$ for a.e. $(x,t)\in\Omega \times [0,T]$ for any $T<\infty$.
\end{lem}
\begin{proof}
	By using the equation \eqref{eq-app}, we apply a simple computation to show that
	\begin{equation}\label{eq-poin-es}
		\begin{aligned}
			\frac{1}{2}\frac{\p}{\p t}\int_{|u^\ep|>1}(|u^\ep|-1)^2dx=&\int_{\Om}\<(|u^\ep|-1)^+, \p_t (|u^\ep|-1)^+\>dx\\
			=&\int_{|u^\ep|>1}\<|u^\ep|-1, \frac{\<\p_t u^\ep, u^\ep\>}{|u^\ep|}\>dx\\
			=&\int_{|u^\ep|>1}\<\p_t u^\ep,\frac{(|u^\ep|-1)u^\ep}{|u^\ep|}\>dx\\	
			=&-\int_{|u^\ep|>1}\<\n_vu^\ep,\frac{(|u^\ep|-1)u^\ep}{|u^\ep|}\>dx\\
			&+\ep\int_{|u^\ep|>1}\<\De u_\ep,\frac{(|u^\ep|-1)u^\ep}{|u^\ep|}\>dx=J_1+J_2.
		\end{aligned}
	\end{equation}
	
	To proceed, we need to show the precise formula of $J_1$ and $J_2$ respectively. By applying integration by parts, we can show
	\begin{align*}
		J_1=&-\int_{\Om}v\cdot \frac{\<\n u^\ep, u^\ep\>}{|u^\ep|}(|u^\ep|-1)^+dx\\
		=&-\int_{\Om}v\cdot \n(|u^\ep|-1)^+(|u^\ep|-1)^+dx\\
		=&-\frac{1}{2}\int_{\Om}v\cdot \n[(|u^\ep|-1)^+]^2dx=0,
	\end{align*}
	since $\mbox{div}(v)=0$ and $\<v,\nu\>_{\p\Om}=0$.
	
	Next, we can utilize a similar argument as that for $J_1$ to get the precise formula of $J_2$ as follows.
	\begin{align*}
		J_2=&\ep\int_{|u^\ep|>1}\<\De u_\ep,\frac{(|u^\ep|-1)u^\ep}{|u^\ep|}\>dx\\
		=&-\ep\int_{|u^\ep|>1}\frac{\<\n u^\ep,u^\ep\>^2}{|u^\ep|^3}dx\\
		&-\ep\int_{|u^\ep|>1}\frac{(|u^\ep|-1)|\n u^\ep|^2}{|u^\ep|}dx.
	\end{align*}
	
	By substituting the equations for $J_1$ and $J_2$ into \eqref{eq-poin-es}, we have
	\[\frac{\p}{\p t}\int_{|u^\ep|>1}(|u^\ep|-1)^2dx\leq 0.\]
	This means that the following function
	\begin{equation*}
		q(t)=\int_{|u^\ep|>1}(|u^\ep|-1)^2dx
	\end{equation*}
	is decreasing non-negative function. Noting $|u_0|=1$, i.e. $q(0)=0$, we get that $q(t)\equiv0$ for any $t> 0$. Therefore, we have $|u^\varepsilon|\leq 1$ a.e. $(x,t)\in \Omega\times[0,T]$.
\end{proof}

Using the above lemma, we have $|u^\varepsilon|\leq 1$ a.e. $(x,t)\in\Omega \times(0,\infty)$. Hence $u^\varepsilon$ is a strong solution of the following equation
\begin{equation}\label{eq-app1}
\begin{cases}
\partial_t u^\varepsilon+ v\cdot \nabla u^\varepsilon=\varepsilon\Delta u^\varepsilon+u^\varepsilon\times \Delta u^\varepsilon, &(x,t)\in \Om\times(0,T)\\[1ex]
\frac{\partial u^\varepsilon}{\partial \nu}|_{\partial \Omega\times (0,T)}=0,\\[1ex]
u^\varepsilon(x,0)=u_0:\Om \to \U^2,
\end{cases}
\end{equation}
where $v$ satisfies that $\mbox{div}(v)=0$ in $\Om\times (0,\infty)$ and $\<v,\nu\>=0$ on $\p\Om\times (0,\infty)$.

Then we can get the following uniform energy estimates for $u^\ep$ with respect to $\ep$.
\begin{lem}\label{es-u-ep}
For the solution $u^\ep$, the following properties hold true.
\begin{itemize}
\item[$(1)$] For any $T<\infty$, there holds a priori estimate for $u^\ep$:
\begin{equation}\label{eq4.17}
	\sup_{0\leq t\leq T} \|u^\varepsilon\|_{H^1(\Omega)}^2\leq \exp(I(T))\int_\Omega |\nabla u_0|^2dx+\int_{\Om}|u_0|^2dx.
\end{equation}

\item[$(2)$] For any $T<\infty$, $\p_tu^\ep$ satisfies 
\begin{equation}\label{es-ptu}
\norm{\p_tu^\ep}_{L^2([0,T], H^{-1}(\Om))}\leq (2T^{1/2}\exp(1/2I(T))\norm{\n u_0}_{L^2}+S^{1/2}(T)\norm{u_0}_{L^2}).
\end{equation}
\end{itemize}
\end{lem}
\begin{proof}
The first estimate \eqref{eq4.17} is obtained directly by apply the lower semicontinuity of \eqref{eq4.1} and \eqref{eq4.2} respectively, when $n\to \infty$.

Next, we show the uniform estimate \eqref{es-ptu} of $\p_t u^\ep$. For any $\varphi\in C^\infty(\bar{\Om}\times [0,T])$, a simple calculation gives
\begin{align*}
&|\int_{0}^T\int_{\Om}\<\p_t u^\ep, \varphi\>dxdt|\\
\leq &|\int_{0}^T\int_{\Om}\<\ep\n  u^\ep+u^\ep\times \n u^\ep, \n\varphi\>dxdt|+|\int_{0}^T\int_{\Om}\<v\cdot \n u^\ep, \varphi\>dxdt|\\
\leq &(1+\ep^2) T^{1/2}\sup_{0<t<T}\norm{\n u^\ep}_{L^2}\norm{\varphi}_{L^2([0,T],H^1(\Om))}+|\int_{0}^T\int_{\Om}\<v\cdot \n \varphi, u^\ep\>dxdt|\\
\leq & \((1+\ep^2) T^{1/2}\sup_{0<t<T}\norm{\n u^\ep}_{L^2}+S^{1/2}(T)\norm{u_0}_{L^2}\)\norm{\varphi}_{L^2([0,T],H^1(\Om))}\\
\leq &(2T^{1/2}\exp(1/2I(T))\norm{\n u_0}_{L^2}+S^{1/2}(T)\norm{u_0}_{L^2})\norm{\varphi}_{L^2([0,T],H^1(\Om))}.
\end{align*}

Therefore, the desired estimate for $\p_t u^\ep$ can be derived from the above formula directly.
\end{proof}

Next, we will prove the main theorem \ref{th2}.
\begin{proof}[\bf{The proof of the Theorem \ref{th2}}]
 The proof is divided into two steps.
 
\medskip
\noindent\emph{Step 1: The convergence of $u^\ep$ and the limiting map.}\ 
 
For any $T<\infty$, Lemma \ref{es-u-ep} implies that 
\begin{equation}\label{eq4.18}
\norm{u^\varepsilon}_{L^\infty([0,T],H^1(\Omega))}+ \norm{\p_tu^\varepsilon}_{L^{2}([0,T],H^{-1}(\Omega))}\leq C,
\end{equation}
for some constant $C$ independent of $\ep$.
	
Then, there exists a $u\in L^\infty([0,T],H^1(\Omega))$ such that
\begin{align*}
&u^\varepsilon \rightharpoonup u,~~~ weakly^\ast ~in ~~L^\infty([0,T],H^1(\Omega)),~~ as ~\varepsilon\rightarrow0,\\
&\n u^\varepsilon \rightharpoonup u,~~~ weakly
 ~in ~~L^2([0,T],L^2(\Omega)),~~ as ~\varepsilon\rightarrow0.
\end{align*}
Let $X=H^1(\Omega)$, $B=L^2(\Omega)$, $Y=H^{-1}(\Omega)$ in Aubin-Lions-Simon compact lemma \ref{A-S}, we have
\begin{equation*}
u^\varepsilon \rightarrow u,~~strongly~~in ~~C^0([0,T],L^2(\Omega)),
\end{equation*}
Moreover, we have
\begin{equation*}
u^\varepsilon \rightarrow u,~~ a.e. ~(x,t)\in \Omega\times[0,T].
\end{equation*}

We then show that the limiting map $u$ satisfies $|u|=1$. Choosing $u_\varepsilon$ be a test function for \eqref{eq-app1} and using the fact that $|u_0|=1$, we know
\begin{equation*}
	\int_\Omega |u^\varepsilon|^2 dx+\varepsilon\int_0^T\int_\Omega |\nabla u^\varepsilon|^2 dxdt=\int_\Omega|u_0|^2 dx=Vol(\Omega),
\end{equation*}
for a.e. $t\in[0,T]$. As $\varepsilon\rightarrow0$, there holds that
\begin{equation*}
	\int_\Omega(|u|^2-1)dx=0,
\end{equation*}
which implies $|u|=1$ for a.e. $(x,t)\in \Omega\times [0,T]$.
	
\medskip
\noindent\emph{Step 2: The limiting map is a global weak solution to \eqref{eq-ISMF}.}\ 	

For any $\varphi \in C^\infty(\bar{\Omega}\times [0,T])$, the solution $u^\ep$ satisfies
\begin{equation*}
\begin{split}
&\int_0^T\int_\Omega\left\langle\partial_t u^\varepsilon, \varphi\right\rangle dxdt+\int_0^T\int_\Omega\left\langle v\cdot \nabla u^\varepsilon, \varphi\right\rangle dxdt\\
=&\int_0^T\int_\Omega \left\langle u^\varepsilon\times \Delta u^\varepsilon, \varphi\right\rangle dxdt+\varepsilon \int_0^T\int_\Omega \left\langle\Delta u^\varepsilon,\varphi\right\rangle  dxdt.
\end{split}
\end{equation*}

Since $u^\varepsilon \rightarrow u$ strongly in $C^0([0,T],L^2(\Omega))$ and $u^\ep(x,0)=u_0$, we know that
\begin{align*}
\int_0^T\int_\Omega\left\langle\partial_t u^\varepsilon, \varphi\right\rangle dxdt=&\int_{\Om}\<u^\ep, \phi\>dx(T)-\int_{\Om}\<u_0,\varphi\>(0)-\int_{0}^{T}\<u^\ep, \p_t \phi\>\\
\rightarrow&\int_\Omega\langle u, \varphi\rangle dx(T)-\int_\Omega\langle u_0, \varphi\rangle dx(0)-\int_0^T\int_\Omega\left\langle u, \partial_t\varphi\right\rangle dxdt,
\end{align*}
as $\ep\to 0$.

Using the convergence results for $u^\ep$ in step 1, we can show
\begin{align*}
\varepsilon \int_0^T\int_\Omega \left\langle\Delta u^\varepsilon,\varphi\right\rangle  dxdt=\varepsilon \int_0^T\int_\Omega \left\langle\nabla u^\varepsilon,\nabla\varphi\right\rangle  dxdt\rightarrow0,~~~as ~~\varepsilon\rightarrow 0,
\end{align*}
\begin{align*}
\int_0^T\int_\Omega\left\langle v\cdot \nabla u^\varepsilon, \varphi\right\rangle dxdt=&-\int_0^T\int_\Omega\<u^\ep, v\cdot \n \varphi\>dxdt\\
\rightarrow &-\int_0^T\int_\Omega\left\langle u, v\cdot \nabla \varphi \right\rangle dxdt\\
=&\int_0^T\int_\Omega\left\langle v\cdot \nabla u, \varphi\right\rangle dxdt,~~~as ~~\varepsilon\rightarrow0,
\end{align*}
and
\begin{align*}
\int_0^T\int_\Omega \left\langle u^\varepsilon\times \Delta u^\varepsilon, \varphi\right\rangle dxdt=&-\int_0^T\int_\Omega \left\langle u^\varepsilon\times \n u^\varepsilon, \n \varphi\right\rangle dxdt\\
\rightarrow& -\int_0^T\int_\Omega \left\langle u\times \nabla u, \nabla \varphi\right\rangle dxdt,~~~as ~~\varepsilon\rightarrow0.
\end{align*}

To summarize the above arguments, we conclude that the limiting map $u$ satisfies the following equation
\begin{equation*}
\begin{split}
&\int_\Omega\langle u, \varphi\rangle dx(T)-\int_\Omega\langle u_0, \varphi\rangle dx(0)-\int_0^T\int_\Omega\left\langle u, \partial_t\varphi\right\rangle dxdt\\
=&-\int_0^T\int_\Omega\left\langle v\cdot \nabla u, \varphi\right\rangle dxdt-\int_0^T\int_\Omega \left\langle u\times \nabla u, \nabla \varphi\right\rangle dxdt
\end{split}
\end{equation*}
for any $\phi\in C^\infty(\bar{\Om}\times [0,T])$ and any $T<\infty$. By the similar argument with that in Theorem \ref{lm9}, we can prove that
\begin{equation*}
	\frac{\partial u}{\partial \nu}=0,~~~~ (x,t)\in \partial\Omega \times[0,T],
\end{equation*}
in the sense of distribution. 

Therefore we complete the proof of the theorem.
\end{proof}

\section{Local regular solutions}\label{s: reg-solu}

\subsection{Local regular solution to parabolic perturbed equation}
In this subsection, we consider the following initial-Neumann boundary value problem of the approximation equation for the incompressible Schr\"odinger flow \eqref{eq-ISMF}
\begin{equation}\label{eq-AISMF}
\begin{cases}
\p_tu=\ep\tau_v(u)+u\times\tau_v(u),\quad\quad&\text{(x,t)}\in\Om\times \Real^+,\\[1ex]
\frac{\p u}{\p \nu}=0, &\text{(x,t)}\in\p\Om\times \Real^+,\\[1ex]
u(x,0)=u_0: \Om\to \U^2,
\end{cases}
\end{equation}
where $\Om$ is a bounded smooth domain in $\Real^3$, and we set
\[\tau_v(u)=\tau(u)+u\times\n_v u=\De u+|\n u|^2u+u\times\n_vu.\]

Assume that $u_0\in H^3(\Om, \U^2)$ with $\frac{\p u_0}{\p \nu}|_{\p\Om}=0$, we recall that the local existence of regular solutions
to \eqref{eq-AISMF} has been established in \cite{CW2} (also see \cite{CJ}), which can be presented as follows.
\begin{thm}\label{mth1}
Let $\Om$ be a smooth bounded domain in $\Real^3$. Let $u_0\in H^{3}(\Om)$ satisfy the compatibility condition:
	\[\frac{\p u_0}{\p \nu}|_{\p \Om}=0.\]
Suppose that $v\in L^\infty(\Real^+,W^{1,3}(\Om))\cap C^0(\Real^+, H^1(\Om))$, $\p_t v\in L^2(\Real^+, H^1(\Om))$, $\textnormal{\mbox{div}}(v)=0$ inside $\Om$ for any $t\in\Real^+$ and $\<v,\nu\>|_{\p\Om\times \Real^+}=0$. Then there exists a constant $T_\ep>0$ depending only on the $\ep$, $\norm{u_0}_{H^2(\Om)}$ and $\norm{v}_{L^\infty(\Real^+,W^{1,3}(\Om))}$ such that \eqref{eq-AISMF} admits a unique local solution $u_\ep$, for any $T<T_\ep$ which satisfies
\begin{equation}\label{es-app-sol}
\p^i_t u_\ep\in C^0([0,T], H^{3-2i}(\Om))\cap L^2([0,T], H^{4-2i}(\Om))
\end{equation}
for $i=0,1$.
\end{thm}

We then use the solution $u_\ep$ to \eqref{eq-AISMF} that we have obtained in Theorem \ref{mth1} to approximate a regular solution to \eqref{eq-ISMF}. The key point of this progress is to show uniform $W^{3,2}$-energy estimates for $u_\ep$ with respect to $\ep$. To this end, we need to demonstrate some crucial properties for the approximation solution $u_\ep$, which are stated as the following lemmas.

\begin{lem}\label{Equ-norm}
Under the same assumption as that given in Theorem \ref{mth1}, the solution $u_\ep$ satisfies the following properties.
\begin{itemize}
\item[$(1)$] For a.e. $(x,t)\in \Om\times[0,T_\ep)$, we have
\begin{equation}\label{form-1}
\De u_\ep=\frac{1}{1+\ep^2}(\ep \p_t u_\ep-u_\ep\times \p_t u_\ep)-|\n u_\ep|^2u_\ep-u_\ep\times\n_vu_\ep.
\end{equation}
\item[$(2)$] There exists a constant $C$ independent of $\ep$ such that there holds
\begin{equation}\label{equ-n}
\norm{u_\ep}^2_{H^3}\leq C(1+\norm{u_\ep}^2_{H^2}+\norm{\frac{\p u_\ep}{\p t}}^2_{H^1}+\norm{v}^2_{W^{1,3}})^3.
\end{equation}
\end{itemize}
\end{lem}
\begin{proof}
The formula in $(1)$ is obtained directly by applying the equation
\[\p_tu=\ep\tau_v(u)+u\times\tau_v(u).\]

Then it remains to prove the inequality \eqref{equ-n} in $(2)$. By utilizing the formula in $(1)$, we have
\begin{align*}
\int_{\Om}|\De u_\ep|^2dx\leq & C\{\int_{\Om}|\p_tu_\ep|^2dx+\int_{\Om}|\n u_\ep|^4+\int_{\Om}|\n u_\ep|^2|v|^2dx\}\\
\leq &C\(\norm{\p_t u_\ep}^2_{L^2}+\norm{u_\ep}^4_{H^2}+\norm{u_\ep}^2_{H^2}\norm{v}^2_{L^3}\).
\end{align*}

On the other hand, we can apply a simple calculation to derive
\begin{align*}
\n\De u_\ep=&\frac{1}{1+\ep^2}(\ep \n\p_t u_\ep-u_\ep\times \n\p_t u_\ep)\\
&-\frac{1}{1+\ep^2}\n u_\ep\times \p_t u_\ep-2\<\n^2 u_\ep,\n u_\ep\>u_\ep\\
&-|\n u_\ep|^2\n u_\ep-\n u_\ep\times\n_vu_\ep-u_\ep\#\n u_\ep\#\n v-u_\ep\#\n^2 u_\ep\#v,
\end{align*}
from the formula \eqref{form-1}, where $\#$ denotes the linear contraction. Then by applying a similar argument with that in \cite{CW1}, we can show
\begin{align*}
\int_{\Om}|\n \De u_\ep|^2dx\leq& C\{\int_{\Om}|\n \p_tu_\ep|^2dx+\int_{\Om}|\n u_\ep|^2|\p_t u_\ep|^2dx\}\\
&+C\{\int_{\Om}|\n^2 u_\ep|^2|\n u_\ep|^2dx+\int_{\Om}|\n u_\ep|^6dx\}\\
&+C\{\int_{\Om}|\n u_\ep|^4|v|^2dx+\int_{\Om}|\n u_\ep|^2|\n v|^2dx+\int_{\Om}|\n^2 u_\ep|^2|v|^2dx\}\\
\leq& C(1+\norm{u_\ep}^2_{H^2}+\norm{\p_tu_\ep}^2_{H^1})^3+\frac{1}{4}\norm{\n \De u_\ep}^2_{L^2}+V_1+V_2+V_3.
\end{align*}
Here we have used the following Sobolev embedding
\[L^6(\Om)\hookrightarrow W^{1,2}(\Om),\]
and have applied Lemma \ref{eq-norm} and H\"older inequality to give
\begin{align*}
\int_{\Om}|\n^2 u_\ep|^2|\n u_\ep|^2dx\leq& \norm{\n u_\ep}_{L^2}\norm{\n^2u_\ep}_{L^6}\norm{\n u_\ep}^2_{L^6}\\
\leq &C\norm{u_\ep}_{H^2}(\norm{u_\ep}_{H^2}+\norm{\n \De u_\ep}_{L^2})	\\
\leq &C\norm{u_\ep}^4_{H^2}+C\norm{u_\ep}^6_{H^2}+\frac{1}{4}\norm{\n \De u_\ep}^2_{L^2}.
\end{align*}

Next, we estimate the terms on the right hand side of the above inequality as follows.
\begin{align*}
V_1=&C\int_{\Om}|\n u_\ep|^4|v|^2dx\leq C\norm{\n u_\ep}^4_{L^6}\norm{v}^2_{6}\leq C\norm{ u_\ep}^4_{H^2}\norm{v}^2_{H^1},\\
V_2=&C\int_{\Om}|\n u_\ep|^2|\n v|^2dx\leq C\norm{\n u_\ep}^2_{L^6}\norm{\n v}^2_{L^3}\leq C\norm{ u_\ep}^2_{H^2}\norm{v}^2_{W^{1,3}},\\
V_3=&C\int_{\Om}|\n^2 u_\ep|^2|v|^2dx\leq C\norm{\n^2 u_\ep}_{L^2}\norm{\n^2 u_\ep}_{L^6}\norm{v}^2_{L^6}\\
\leq &C\norm{v}_{H^1}\norm{u_\ep}_{H^2}(\norm{u_\ep}_{H^2}+\norm{\n \De u_\ep}_{L^2})\\
\leq & C\norm{v}^2_{H^1}\norm{u_\ep}^2_{H^2}+C\norm{v}^4_{H^1}\norm{u_\ep}^2_{H^2}+\frac{1}{4}\norm{\n \De u_\ep}^2_{L^2}.
\end{align*}
\end{proof}
The above estimates for $V_1$-$V_3$ gives that
\[\norm{\n \De u_\ep}^2_{L^2}\leq C(1+\norm{u_\ep}^2_{H^2}+\norm{\p_tu_\ep}^2_{H^1}+\norm{v}^2_{W^{1,3}})^3.\]

Consequently, by using Lemma \ref{eq-norm}, we can combine the estimates for $\norm{\De u_\ep}_{L^2}$ and $\norm{\n\De u_\ep}_{L^2}$ to show the desired inequality \eqref{equ-n}.

Therefore, the proof is finished.

\medskip
\begin{lem}\label{c-c}
Under the same assumption as that given in Theorem \ref{mth1}, the solution $u_\ep$ satisfies
\[\frac{\p}{\p \nu}\p_tu_\ep|_{\p\Om\times [0,T_\ep)}=0\quad\text{and}\quad \frac{\p}{\p \nu}\tau_v(u_\ep)|_{\p\Om\times[0,T_\ep)}=0\]
in the sense of trace.
\end{lem}
\begin{proof}
Since $u_\ep$ satisfies the Neumann boundary condition
\[\frac{\p u_\ep}{\p\nu}|_{\p\Om\times[0,T]}=0,\]
then for any $\phi\in C^\infty(\bar{\Om}\times[0,T])$ with $T<T_\ep$, there holds
\begin{align*}
\int_{0}^T\int_{\Om}\<\De u_\ep, \p_t\phi\>dxdt=-\int_{0}^T\int_{\Om}\<\n u_\ep, \n\p_t\phi\>dxdt.
\end{align*}

On the other hand, since $u_\ep\in L^2([0,T], H^4(\Om))$ and $\p_tu_\ep\in L^2([0,T], H^2(\Om))$, the embedding lemma \ref{C^0-em} implies
\[u_\ep\in C^0([0,T], H^3(\Om)),\]
which tells us that
\[\frac{\p u(x,t)}{\p \nu}|_{\p\Om}=0,\]
for any $t\in [0,T_\ep)$.

Then, by utilizing the integration by parts, we can apply a simple calculation to show
\begin{align*}
\int_{0}^T\int_{\Om}\<\De u_\ep, \p_t\phi\>dxdt=&-\int_{0}^{T}\int_{\Om}\<\p_t\De u_\ep, \phi\>dxdt+\int_{\Om}\<\De u_\ep, \phi\>dx(T)\\
	&-\int_{\Om}\<\De u_\ep, \phi\>dx(0)\\
	=&-\int_{0}^{T}\int_{\Om}\<\p_t\De u_\ep, \phi\>dxdt-\int_{\Om}\<\n u_\ep, \n\phi\>dx(T)\\
	&+\int_{\Om}\<\n u_\ep, \n\phi\>dx(0)
\end{align*}
and
\begin{align*}
-\int_{0}^{T}\int_{\Om}\<\n u_\ep, \p_t\n \phi\>dxdt
	=&\int_{0}^{T}\int_{\Om}\<\n \p_tu_\ep, \n \phi\>dxdt-\int_{\Om}\<\n u_\ep, \n\phi\>dx(T)\\
	&+\int_{\Om}\<\n u_\ep, \n\phi\>dx(0).
\end{align*}
 Then we can derive from the above two formulae that
\[\int_{0}^{T}\int_{\Om}\<\De \p_tu_\ep, \phi\>dxdt=-\int_{0}^{T}\int_{\Om}\<\n \p_tu_\ep, \n \phi\>dxdt,\]
that is,
\[\frac{\p}{\p\nu}\p_tu_\ep|_{\p\Om\times[0,T]}=0\]
in the sense of trace.

By applying the fact that $\n \tau_v(u_\ep)$ is orthogonal to $u_\ep \times \n \tau_v(u_\ep)$, the equation
\[\p_tu_\ep=\ep\tau_v(u_\ep)+u_\ep\times \tau_v(u_\ep)\]
implies $\frac{\p}{\p \nu}\tau_v(u_\ep)|_{\p\Om\times [0,T_\ep)}=0$ in the sense of trace, since $\frac{\p u_\ep}{\p \nu}|_{\p \Om \times[0,T_\ep)}=0$.
\end{proof}

\begin{lem}
Under the same assumption as that given in Theorem \ref{mth1}, the solution $u_\ep$ satisfies the following properties.
\begin{itemize}
\item[$(1)$] For a.e. $(x,t)\in \Om\times[0,T_\ep)$, there holds
\begin{equation}\label{Key1}
\p_t\p_tu_\ep=\ep\De \p_tu_\ep+u_\ep\times\De\p_tu_\ep+F+L+K,
\end{equation}
where
\begin{align*}
F=&u_\ep\times(u_\ep\times\n_v\p_tu_\ep)+\ep u_\ep\times\n_v\p_tu_\ep
+2\ep\<\n\p_tu_\ep,\n u_\ep\>u_\ep,\\
L=&\p_tu_\ep\times(\De u_\ep+u_\ep \times\n_vu_\ep)+u_\ep\times (\p_tu_\ep\times \n_vu_\ep)\\
&+\ep|\n u_\ep|^2\p_tu_\ep+\ep\p_tu_\ep\times\n_vu_\ep,\\
=&\p_tu_\ep\times(\De u_\ep+u_\ep \times\n_vu_\ep)+\ep|\n u_\ep|^2\p_tu_\ep+\ep\p_tu_\ep\times\n_vu_\ep,\\
K=&u_\ep\times(u_\ep\times\n_{\p_tv}u_\ep)+\ep u_\ep\times\n_{\p_tv}u_\ep.
\end{align*}
\item[$(2)$] For a.e. $(x,t)\in \Om\times[0,T_\ep)$, there also holds
\begin{equation}\label{Key2}
\begin{aligned}
	&\p_t\p_tu_\ep+(1-\ep^2)\De \tau_v(u_\ep)-2\ep\De(u_\ep\times \tau_v(u_\ep))\\
	=&-\ep\{2\n u_\ep\dot{\times}\n\tau_v(u_\ep)+\De u_\ep\times(|\n u_\ep|^2u_\ep+u_\ep\times\n_vu_\ep)\}\\
	&+\ep\{u_\ep\times\n_v\p_tu_\ep+|\n u_\ep|^2\p_tu_\ep+u_\ep\times\n_{\p_tv}u_\ep\}\\
	&+|\n u_\ep|^2\tau_v(u_\ep)-2\<\n u_\ep,\tau_v(u_\ep)\>\cdot\n u_\ep-\n_v\p_tu_\ep\\
	&-\n_{\p_tv}u_\ep+f_1u_\ep+\p_tu_\ep\times f_2,	
\end{aligned}
\end{equation}
where
\begin{align*}
f_1=&\<\De\tau_v(u_\ep), u_\ep\>-\<\p_tu_\ep,\n_vu_\ep\>+2\ep\<\n \p_t u_\ep, \n u_\ep\>,\\
f_2=&\De u_\ep+u_\ep\times\n_vu_\ep+\ep\n_v u_\ep.
\end{align*}
\end{itemize}
\end{lem}
\begin{proof}
By differentiating the both sides of the equation
\[\p_tu_\ep=\ep\tau_v(u_\ep)+u_\ep\times(\De u_\ep+u_\ep\times \n_vu_\ep)\]
in the direction of $t$, where
\[\tau_v(u_\ep)=\tau(u_\ep)+u_\ep\times\n_v u_\ep=\De u_\ep+|\n u_\ep|^2u_\ep+u_\ep\times\n_vu_\ep,\]
we can show
\begin{align*}
\p_t\p_tu_\ep=&\ep\p_t\tau_v(u_\ep)+\p_tu_\ep\times(\De u_\ep+u_\ep\times \n_vu_\ep)\\
&+u_\ep\times(\De \p_tu_\ep+\p_t u_\ep\times\n_vu_\ep+u_\ep\times\n_v\p_tu_\ep+u_\ep\times\n_{\p_tv}u_\ep),
\end{align*}
where
\begin{align*}
\p_t\tau_v(u_\ep)=&\De \p_t u_\ep+|\n u_\ep|^2\p_tu_\ep+2\<\n \p_tu_\ep,\n u_\ep\>u_\ep\\
&+\p_t u_\ep\times\n_vu_\ep+u_\ep\times\n_v\p_tu_\ep+u_\ep\times\n_{\p_tv}u_\ep.
\end{align*}
Then formula \eqref{Key1} follows from the above equation directly.

Next, we intend to use the facts:
\begin{itemize}
	\item[$(1)$] $|u_\ep|=1$,
	\item[$(2)$] The Lagrangian formula: $a\times(b\times c)=\<a,c\>b-\<a,b\>c$,
	\item[$(3)$] The structure of equation: $\p_tu_\ep=\ep\tau_v(u_\ep)+u_\ep\times\tau_v(u_\ep)$,
\end{itemize}
to show that $u_\ep$ satisfies another fourth order differential formula \eqref{Key2}. To proceed, we need to obtain the precise formula of each term in the right hand side of \eqref{Key1} as follows.
\begin{align*}
\ep\De \p_tu_\ep=&\ep^2\De \tau_v(u_\ep)+\ep\De(u_\ep \times\tau_v(u_\ep)).\\
u_\ep\times\De \p_tu_\ep=&\ep u_\ep\times \De \tau_{v}(u_\ep)+u_\ep\times\De(u_\ep\times\tau_v(u_\ep))\\
=&\ep\De(u_\ep\times\tau_v(u_\ep))-2\ep\n u_\ep\dot{\times}\n\tau_v(u_\ep)-\ep\De u_\ep\times(|\n u_\ep|^2u_\ep+u_\ep\times\n_v u_\ep)\\
&+u_\ep\times(u_\ep\times\De\tau_v(u_\ep))+u_\ep\times(\De u_\ep\times\tau_v(u_\ep))+2u_\ep\times(\n u_\ep\dot{\times}\n \tau_v(u_\ep)),
\end{align*}
where
\begin{align*}
u_\ep\times(u_\ep\times\De\tau_v(u_\ep))=&-\De \tau_v(u_\ep)+\<\De \tau_v(u_\ep), u_\ep\>u_\ep,\\
u_\ep\times(\De u_\ep\times\tau_v(u_\ep))=&\<u_\ep,\tau_v(u_\ep)\>\De u_\ep-\<u_\ep,\De u_\ep\>\tau_v(u_\ep)\\
=&-\<u_\ep,\De u_\ep\>\tau_v(u_\ep)=|\n u_\ep|^2\tau_v(u_\ep),\\
2u_\ep\times(\n u_\ep\dot{\times}\n \tau_v(u_\ep))=&2\<u_\ep,\n \tau_v(u_\ep)\>\cdot\n u_\ep-2\<u_\ep,\n u_\ep\>\cdot\n \tau_v(u_\ep)\\
=&2\<u_\ep,\n \tau_v(u_\ep)\>\cdot\n u_\ep=-2\<\n u_\ep,\tau_v(u_\ep)\>\cdot\n u_\ep.
\end{align*}
Hence, we have
\begin{align*}
u_\ep\times\De \p_tu_\ep=&-\De \tau_v(u_\ep)+\ep\De(u_\ep\times\tau_v(u_\ep))\\
&-\ep\{2\n u_\ep\dot{\times}\n\tau_v(u_\ep)+\De u_\ep\times(|\n u_\ep|^2u_\ep+u_\ep\times\n_v u_\ep)\}\\
&+\<\De \tau_v(u_\ep), u_\ep\>u_\ep+|\n u_\ep|^2\tau_v(u_\ep)-2\<\n u_\ep,\tau_v(u_\ep)\>\cdot\n u_\ep.
\end{align*}

Next we obtain the more finer formulae of $F$, $L$ and $K$. A simple computation gives that
\begin{align*}
F=&u_\ep\times(u_\ep\times\n_v\p_tu_\ep)+\ep u_\ep\times\n_v\p_tu_\ep
+2\ep\<\n\p_tu_\ep,\n u_\ep\>u_\ep\\
=&-\n_v\p_tu_\ep+\<\n_v\p_tu_\ep,u_\ep\>u_\ep+\ep u_\ep\times\n_v\p_tu_\ep
+2\ep\<\n\p_tu_\ep,\n u_\ep\>u_\ep\\
=&-\n_v\p_tu_\ep-\<\p_tu_\ep,\n_vu_\ep\>u_\ep+\ep u_\ep\times\n_v\p_tu_\ep
+2\ep\<\n\p_tu_\ep,\n u_\ep\>u_\ep,
\end{align*}
since
\[\<\n_v\p_tu_\ep,u_\ep\>=-\<\p_tu_\ep,\n_vu_\ep\>;\]
\begin{align*}
L=&\p_tu_\ep\times(\De u_\ep+u_\ep \times\n_vu_\ep)+\ep|\n u_\ep|^2\p_tu_\ep+\ep\p_tu_\ep\times\n_vu_\ep,
\end{align*}
since
\[u_\ep\times (\p_tu_\ep\times \n_vu_\ep)=\<u_\ep,\n_vu_\ep\>\p_t u_\ep-\<u_\ep,\p_tu_\ep\>\n_v u_\ep=0;\]
and
\begin{align*}
	K=&-\n_{\p_tv}u_\ep+\ep u_\ep\times\n_{\p_tv}u_\ep.
\end{align*}

Therefore, by combining the above equations with \eqref{Key1}, we get the fourth order differential formula \eqref{Key2}.
\end{proof}

\subsection{$H^1$-energy estimate}
Now we are in the position to show the uniform energy estimates for $u_\ep$. First of all, we can take $u_\ep$ as a test function for \eqref{eq-AISMF} to give
\[\frac{1}{2}\frac{\p}{\p t}\int_{\Om}|u_\ep|^2dx=0.\]
Then utilizing $-\De u_\ep$ as another test function, we have
\begin{equation}
\begin{aligned}
&\frac{1}{2}\frac{\p}{\p t}\int_{\Om}|\n u_\ep|^2dx+\ep\int_{\Om}|u_\ep\times \De u_\ep|^2dx\\
=&-\ep\int_{\Om}\<u_\ep\times\n_v u_\ep,\De u_\ep\>dx+\int_{\Om}\<\n_vu_\ep, \De u_\ep\>dx\\
\leq& \frac{\ep}{2}\int_{\Om}|u_\ep\times \De u_\ep|^2dx+C\ep\norm{v}^2_{L^3}\norm{\n u_\ep}^2_{L^6}\\
&+\int_{\Om}|\n v||\n u_\ep|^2dx+|\int_{\Om}v\cdot\<\n^2u_\ep, \n u_\ep\>dx\\
\leq&\frac{\ep}{2}\int_{\Om}|u_\ep\times \De u_\ep|^2dx+C(\ep\norm{v}^2_{L^3}+\norm{\n v}_{L^3})\norm{\n u_\ep}^2_{H^1},
\end{aligned}
\end{equation}
where we have used the facts $\mbox{div}(v)=0$ with $\<v,\nu\>|_{\p\Om}=0$ to show
\[\int_{\Om}v\cdot\<\n^2u_\ep, \n u_\ep\>dx=\frac{1}{2}\int_{\Om}v\cdot\n|\n u_\ep|^2dx=0.\]

This yields that
\begin{equation}\label{es-H1-new}
\frac{\p}{\p t}\int_{\Om}|\n u_\ep|^2dx+\ep\int_{\Om}|u_\ep\times \De u_\ep|^2dx\leq C(\ep\norm{v}^2_{L^3}+\norm{\n v}_{L^3})\norm{\n u_\ep}^2_{H^1}.
\end{equation}

\subsection{$H^2$-energy estimate}
Taking $\p_t u_\ep$ as a text function for formula \eqref{Key2}, we apply a direct computation to give
\begin{equation}\label{eq-H2}
\begin{aligned}
&\frac{1}{2}\frac{\p}{\p t}\int_{\Om}|\p_t u_\ep|^2dx+(1-\ep^2)\int_{\Om}\<\De \tau_v(u_\ep),\p_t u_\ep\>dx-2\ep\int_{\Om}\<\De (u_\ep\times\tau_v(u_\ep)),\p_tu_\ep\>dx\\
=&-\ep\{\int_{\Om}\<2\n u_\ep\dot{\times}\n\tau_v(u_\ep)+\De u_\ep\times(|\n u_\ep|^2u_\ep+u_\ep\times\n_vu_\ep),\p_t u_\ep\>dx\}\\
&+\ep\{\int_{\Om}\<u_\ep\times\n_v\p_tu_\ep+|\n u_\ep|^2\p_tu_\ep+u_\ep\times\n_{\p_tv}u_\ep,\p_t u_\ep\>dx\}\\
&+\int_{\Om}\<|\n u_\ep|^2\tau_v(u_\ep)-2\<\n u_\ep,\tau_v(u_\ep)\>\cdot\n u_\ep-\n_v\p_tu_\ep-\n_{\p_tv}u_\ep,\p_t u_\ep\>dx\\
=&\ep(I_1+I_2+I_3+I_4+I_5+I_6)+(II_1+II_2+II_3+II_4),
\end{aligned}
\end{equation}
where we have used the facts that $\<u_\ep, \p_t u_\ep\>=0$ and $\<\p_tu_\ep\times \cdot, \p_t u_\ep\>=0$ to show
\[\int_{\Om}\<f_1u_\ep,\p_t u_\ep\>dx=0,\quad \text{and}\quad\int_{\Om}\<\p_tu_\ep\times f_2,\p_t u_\ep\>dx=0.\]

By applying Lemma \ref{c-c}, we can estimate the second term on the lift hand side of \eqref{eq-H2} as follows.
\begin{align*}
&(1-\ep^2)\int_{\Om}\<\De \tau_v(u_\ep),\p_t u_\ep\>dx\\
=&(1-\ep^2)\int_{\Om}\<\tau_v(u_\ep),\De \p_t u_\ep\>dx\\
=&\frac{1-\ep^2}{2}\frac{\p}{\p t}\int_{\Om}|\De u_\ep|^2dx-(1-\ep^2)\int_{\Om}\<\n(|\n u_\ep|^2u_\ep+u_\ep\times \n_v u_\ep), \n \p_t u_\ep\>dx\\
=&\frac{1-\ep^2}{2}\frac{\p}{\p t}\int_{\Om}|\De u_\ep|^2dx+III_1+III_2.
\end{align*}
Here,
\begin{align*}
|III_1|\leq& (1-\ep^2)\int_{\Om}2|\n^2 u_\ep||\n u_\ep||\n\p_t u_\ep|+|\n u_\ep|^3|\n\p_tu_\ep|dx\\
\leq & C\norm{\n u_\ep}_{L^6}\norm{\n^2 u_\ep}_{L^3}\norm{\n\p_t u_\ep}_{L^2}+C\norm{\n u_\ep}^3_{L^6}\norm{\n\p_t u_\ep}_{L^2}\\
\leq&C(\norm{u_\ep}_{H^2}\norm{u_\ep}_{H^3}+\norm{u_\ep^2}^3_{H^2})\norm{\n\p_t u_\ep}_{L^2}\\
\leq&C(1+\norm{u_\ep}^2_{H^2}+\norm{\p_t u_\ep}^2_{H^1}+\norm{v}^2_{W^{1,3}})^3,
\end{align*}
where we have applied Lemma \ref{eq-norm} in the last line in above estimate of $III_1$,
\begin{align*}
|III_2|=&(1-\ep^2)|\int_{\Om}\<\n(u_\ep\times \n_v u_\ep), \n \p_t u_\ep\>dx|\\
\leq&C\int_{\Om}|\n u_\ep|^2|v||\n \p_t u_\ep|+|\n u_\ep||\n v||\n\p_t u_\ep|+|v||\n^2 u_\ep||\n \p_t u_\ep|dx\\
\leq&C(\norm{\n u_\ep}^2_{L^6}\norm{v}_{L^6}+\norm{\n v}_{L^3}\norm{\n u_\ep}_{L^6}+\norm{\n^2 u_\ep}_{L^3}\norm{v}_{L^6})\norm{\n \p_t u_\ep}_{L^2}\\
\leq &C(1+\norm{u_\ep}^2_{H^2}+\norm{\p_t u_\ep}^2_{H^1}+\norm{v}^2_{W^{1,3}})^3.
\end{align*}

Since $u_\ep\times\tau_v(u_\ep)=\p_t u_\ep-\ep\tau_v(u_\ep)$, the last term on the lift hand side of \eqref{eq-H2} can be controlled in below.
\begin{align*}
-2\ep\int_{\Om}\<\De(u_\ep\times \tau_v(u_\ep)),\p_tu_\ep\>dx=&-2\ep\int_{\Om}\<u_\ep\times \tau_v(u_\ep),\De \p_tu_\ep\>dx\\
=&2\ep^2\int_{\Om}\<\tau_v(u_\ep),\De \p_tu_\ep\>dx-2\ep\int_{\Om}\<\p_t u_\ep,\De \p_tu_\ep\>dx\\
\geq &\ep^2\frac{\p}{\p t}\int_{\Om}|\De u_\ep|^2dx+2\ep\int_{\Om}|\n \p_tu_\ep|^2dx\\
&-C\ep^2(1+\norm{u_\ep}^2_{H^2}+\norm{\p_t u_\ep}^2_{H^1}+\norm{v}^2_{W^{1,3}})^3.
\end{align*}

Consequently, we can combine the above estimates to show
\begin{equation}\label{LHS}
\begin{aligned}
\text{LHS of \eqref{eq-H2}}\geq& \frac{1}{2}\frac{\p}{\p t}\int_{\Om}|\p_t u_\ep|^2dx+\frac{1+\ep^2}{2}\frac{\p}{\p t}\int_{\Om}|\De u_\ep|^2dx\\
&+2\ep\int_{\Om}|\n \p_tu_\ep|^2dx-C(1+\norm{u_\ep}^2_{H^2}+\norm{\p_t u_\ep}^2_{H^1}+\norm{v}^2_{W^{1,3}})^3.
\end{aligned}
\end{equation}

Next, we demonstrate the estimates of the ten terms on the right hand side of \eqref{eq-H2} respectively.
\begin{align*}
\ep|I_1|=&\ep|\int_{\Om}\<2\n u_\ep\dot{\times}\n\tau_v(u_\ep),\p_t u_\ep\>dx|\\
\leq&C\ep\int_{\Om}|\n u_\ep||\n \p_tu_\ep||\p_t u_\ep|+|\n u_\ep|^2|\p_t u_\ep|^2dx\\
\leq&C\ep \norm{\n u_\ep}_{L^3}\norm{\p_t u_\ep}_{L^6}\norm{\n \p_t u_\ep}_{L^2}+\ep C\norm{\n u_\ep}^2_{L^4}\norm{\p_t u_\ep}^2_{L^4}\\
\leq&C\ep(1+\norm{u_\ep}^2_{H^2}+\norm{\p_t u_\ep}^2_{H^1})^2,
\end{align*}
where we have used $\tau_v(u_\ep)=\frac{1}{1+\ep^2}(\ep\p_t u_\ep-u_\ep\times \p_t u_\ep)$.
\begin{align*}
\ep|I_2|=&\ep|\int_{\Om}\<\De u_\ep\times(|\n u_\ep|^2u_\ep),\p_t u_\ep\>dx|
\leq C\ep\norm{\De u_\ep}_{L^2}\norm{\n u_\ep}^2_{L^6}\norm{\p_t u_\ep}_{L^6}\\
\leq&C\ep(1+\norm{u_\ep}^2_{H^2}+\norm{\p_t u_\ep}^2_{H^1})^2.\\
\ep|I_3|=&\ep|\int_{\Om}\<\De u_\ep\times(u_\ep\times\n_vu_\ep),\p_t u_\ep\>dx|\\
=&\ep|\int_{\Om}\<\De u_\ep,\n_vu_\ep\>\<u_\ep,\p_t u_\ep\>-\<\De u_\ep,u_\ep\>\<\n_vu_\ep,\p_t u_\ep\>dx|\\
=&\ep|\int_{\Om}|\n u_\ep|^2\<\n_vu_\ep,\p_t u_\ep\>dx|\leq C\ep\norm{\n u_\ep}^3_{L^6}\norm{v}_{L^3}\norm{\p_tu_\ep}_{L^6}\\
\leq& C\ep(1+\norm{u_\ep}^2_{H^2}+\norm{\p_t u_\ep}^2_{H^1}+\norm{v}^2_{L^3})^3.
\end{align*}
\begin{align*}
\ep|I_4|=&\ep|\int_{\Om}\<u_\ep\times\n_v\p_tu_\ep,\p_t u_\ep\>dx|\leq \ep\norm{\n \p_t u_\ep}_{L^2}\norm{v}_{L^3}\norm{\p_tu_\ep}_{L^6}\\
\leq&C\ep(1+\norm{\p_t u_\ep}^2_{H^1}+\norm{v}^2_{L^3})^2.\\
\ep|I_5|=&\ep\int_{\Om}|\n u_\ep|^2|\p_tu_\ep|^2dx\leq \ep\norm{\n u_\ep}^2_{L^4}\norm{\p_tu_\ep}^2_{L^4}\\
\leq&C\ep(1+\norm{u_\ep}^2_{H^2}+\norm{\p_t u_\ep}^2_{H^1})^2.\\
\ep|I_6|=&\ep|\int_{\Om}\<u_\ep\times\n_{\p_tv}u_\ep,\p_t u_\ep\>dx|\leq \ep\norm{\p_t v}_{L^2}\norm{\n u_\ep}_{L^3}\norm{\p_t u_\ep}_{L^6}\\
\leq&C\ep(1+\norm{u_\ep}^2_{H^2}+\norm{\p_t u_\ep}^2_{H^1})^2+C\ep\norm{\p_tv}^2_{L^2}.
\end{align*}
By using the same arguments as that for $I_6$, we have
\[|II_4|\leq C(1+\norm{u_\ep}^2_{H^2}+\norm{\p_t u_\ep}^2_{H^1})^2+\norm{\p_tv}^2_{L^2}.\]

Since $\tau_v(u_\ep)=\frac{1}{1+\ep^2}(\ep\p_t u_\ep-u_\ep\times \p_t u_\ep)$, we can show
\begin{align*}
|II_1|=&|\int_{\Om}\<|\n u_\ep|^2\tau_v(u_\ep),\p_t u_\ep\>dx|=\frac{\ep}{1+\ep^2}\int_{\Om}|\n u_\ep|^2|\p_t u_\ep|^2dx\\
\leq &C\ep\norm{\n u_\ep}^2_{L^4}\norm{\p_t u_\ep}^2_{L^4}\leq C\ep(1+\norm{u_\ep}^2_{H^2}+\norm{\p_t u_\ep}^2_{H^1})^2.
\end{align*}
By applying same arguments as that for term $II_1$, we can get a bound of the term $II_2$
\[|II_2|=|\int_{\Om}\<2\<\n u_\ep,\tau_v(u_\ep)\>\cdot\n u_\ep,\p_t u_\ep\>dx|\leq C(1+\norm{u_\ep}^2_{H^2}+\norm{\p_t u_\ep}^2_{H^1})^2.\]
For terms $II_3$, there holds
\begin{align*}
|II_3|=&|\int_{\Om}\<\n_v\p_tu_\ep,\p_t u_\ep\>dx|\leq \norm{\n \p_tu_\ep}_{L^2}\norm{v}_{L^3}\norm{\p_t u_\ep}_{L^6}\\
\leq&C(1+\norm{\p_t u_\ep}^2_{H^1}+\norm{v}^2_{L^3})^2.
\end{align*}

By substituting the inequality \eqref{LHS} and the estimates for $\ep I_1$-$\ep I_6$ and $II_2$-$II_4$ into \eqref{eq-H2}, there holds
\begin{equation}\label{es-H2-new}
\begin{aligned}
&\frac{1}{2}\frac{\p}{\p t}\int_{\Om}|\p_t u_\ep|^2dx+\frac{1+\ep^2}{2}\frac{\p}{\p t}|\De u_\ep|^2dx+2\ep\int_{\Om}|\n \p_t u_\ep|^2dx\\
\leq &C(1+\norm{u_\ep}^2_{H^2}+\norm{\p_t u_\ep}^2_{H^1}+\norm{v}^2_{W^{1,3}})^3+C\norm{\p_t v}^2_{L^2}.
\end{aligned}
\end{equation}

\medskip
\subsection{Uniform $H^3$-estimate}
To get a uniform $H^3$- estimate of $u_\ep$, we need to improve the regularity of $\p_tu_\ep$ to guarantee that the following energy estimates make sense.

By Theorem \ref{mth1}, the solution $u_\ep$ satisfies
\begin{align*}
\p_t^iu_\ep\in C^0([0,T_\ep), H^{3-2i}{\Om})\cap L^2_{loc}([0,T_\ep), H^{4-2i}(\Om))
\end{align*}
for $i=0,1$. In particular, $\p_t u_\ep\in C^0([0,T_\ep), H^1(\Om))\cap L^2_{loc}([0,T_\ep), H^2(\Om))$ is a strong solution to the following linear equation
\begin{equation}\label{eq-u-ep-t}
\begin{cases}
\p_t \om-\ep\De \om-u_\ep\times \De \om=g,&(x,t)\in \Om\times [0,T_\ep),\\[1ex]
\frac{\p \om}{\p \nu}=0, &(x,t)	\in \p\Om\times [0,T_\ep),\\[1ex]
\om(x,0)=\ep\tau_{v_0}(u_0)+u_0\times \tau_{v_0}(u_0).
\end{cases}
\end{equation}
where $v_0=v(x,0)$, $g=F+L+K$ with
\begin{align*}
	F=&-\n_v\p_tu_\ep-\<\p_tu_\ep,\n_vu_\ep\>u_\ep+\ep u_\ep\times\n_v\p_tu_\ep
	+2\ep\<\n\p_tu_\ep,\n u_\ep\>u_\ep,\\
	L=&\p_tu_\ep\times(\De u_\ep+u_\ep \times\n_vu_\ep)+\ep|\n u_\ep|^2\p_tu_\ep+\ep\p_tu_\ep\times\n_vu_\ep,\\
	K=&-\n_{\p_tv}u_\ep+\ep u_\ep\times\n_{\p_tv}u_\ep,
\end{align*}
and 
\[\tau_{v_0}(u_0)=\ep(\De u_0+|\n u_0|^2u_0+u_0\times \n_{v_0} u_0)+u_0\times \De u_0-\n_{v_0}u_0.\]

On the other hand, under the assumption of $v$ in Theorem \ref{mth1}, that is,
\[v\in L^\infty(\Real^+,W^{1,3}(\Om))\cap C^0(\Real^+, H^1(\Om)),\quad \p_t v\in L^2(\Real^+, H^1(\Om)),\]
it is not difficult to show that
\[g\in L^2_{loc}([0,T_\ep), H^1(\Om)).\]

Then by applying the $L^2$-estimates of parabolic equation (also refer to Theorem A.1 in \cite{CW1}), the above estimate of $g$ implies that
\[\p_t u_\ep\in L^2_{loc}((0,T_\ep), H^3(\Om)),\quad \p^2_t u_\ep\in L^2_{loc}((0,T_\ep), H^1(\Om)).\]
Those regularities of $\p_t u_\ep$ can guarantee that the integration by parts in the following process of energy estimates makes sense.

Now, we demonstrate a uniform $H^3$-bound of $u_\ep$ with respect with $\ep$. Taking $-\De \p_t u_\ep$ as a test function to \eqref{Key1}:
\[\p_t\p_tu_\ep=\ep\De \p_tu_\ep+u_\ep\times\De\p_tu_\ep+F+L+K,\]
one can show
\begin{equation}\label{eq-H3}
\begin{aligned}
&\frac{1}{2}\frac{\p}{\p t}\int_{\Om}|\n \p_t u_\ep|^2dx+\ep\int_{\Om}|\De \p_tu_\ep|^2dx\\
=&-\int_{\Om}\<F+L+K, \De \p_t u_\ep\>dx\\
=&-\ep\int_{\Om}\<u_\ep\times\n_v\p_tu_\ep, \De \p_t u_\ep\>dx+\int_{\Om}\<\n_v\p_tu_\ep, \De \p_t u_\ep\>dx\\
&-\int_{\Om}\<\tilde{F}, \De \p_t u_\ep\>dx-\int_{\Om}\<L, \De \p_t u_\ep\>dx-\int_{\Om}\<K, \De \p_t u_\ep\>dx\\
=&M_1+M_2+M_3+M_4+M_5,
\end{aligned}
\end{equation}
where we denote
\[\tilde{F}=-\<\p_tu_\ep,\n_vu_\ep\>u_\ep+2\ep\<\n\p_tu_\ep,\n u_\ep\>u_\ep.\]

What follows is estimating the above five terms $M_1$-$M_5$ respectively.

\begin{align*}
|M_1|=&\ep|\int_{\Om}\<u_\ep\times\n_v\p_tu_\ep, \De \p_t u_\ep\>dx|\\
\leq &\frac{\ep}{8}\int_{\Om}|\De \p_tu_\ep|^2dx+C\ep\int_{\Om}|v|^2|\n \p_t u_\ep|^2dx\\
\leq &\frac{\ep}{8}\int_{\Om}|\De \p_tu_\ep|^2dx+C\ep\norm{v}^2_{L^\infty}\norm{\n \p_t u_\ep}^2_{L^2}\\
\leq &\frac{\ep}{8}\int_{\Om}|\De \p_tu_\ep|^2dx+C\ep\norm{\p_t u_\ep}^4_{H^1}+\norm{v}^4_{L^\infty}.
\end{align*}
\begin{align*}
|M_2|=&|\int_{\Om}\<\n_v\p_tu_\ep, \De \p_t u_\ep\>dx|=|\int_{\Om}\<\n(\n_v\p_tu_\ep), \n \p_t u_\ep\>dx|\\
\leq &C\int_{\Om}|\n v||\n\p_tu_\ep|^2dx+|\int_{\Om}v\cdot\<\n^2\p_tu_\ep, \n\p_t u_\ep\>dx|\\
\leq &C\norm{\n v}_{L^\infty}\int_{\Om}|\n\p_tu_\ep|^2dx\leq C\norm{\p_t u_\ep}^4_{H^1}+C\norm{\n v}^2_{L^\infty}.
\end{align*}
where we have applied the following fact
\[\int_{\Om}v\cdot\<\n^2\p_tu_\ep, \n\p_t u_\ep\>dx=\frac{1}{2}\int_{\Om}v\cdot\n|\n\p_t u_\ep|^2dx=0\]
since $\mbox{div}(v)=0$ and $\<v, \nu\>|_{\p\Om}=0$.

Due to $\<u_\ep, \p_t u_\ep\>=0$, we have
\[\<u_\ep,\De \p_t u_\ep\>=-\<\De u_\ep,\p_t u_\ep\>-2\<\n u_\ep,\n \p_t u_\ep\>.\]

Then, a simple calculation shows
\begin{align*}
|M_3|\leq& |\int_{\Om}\<\p_tu_\ep,\n_vu_\ep\>\<u_\ep,\De \p_t u_\ep\>dx|+|\int_{\Om}2\ep\<\n\p_tu_\ep,\n u_\ep\>\<u_\ep,\De \p_t u_\ep\>dx|\\
=&a+b,
\end{align*}
where
\begin{align*}
|a|\leq &\int_{\Om}|v||\p_t u_\ep||\n u_\ep|(|\p_t u_\ep||\De u_\ep|+2|\n \p_t u_\ep||\n u_\ep|)dx\\
\leq &C\norm{v}_{L^3}\norm{\p_t u_\ep}^2_{L^6}\norm{\n u_\ep}_{L^6}\norm{\De u_\ep}_{L^6}\\
&+C\norm{v}_{L^6}\norm{\n \p_t u_\ep}_{L^2}\norm{\p_t u_\ep}_{L^6}\norm{\n u_\ep}_{L^6}\norm{\n u_\ep}_{L^\infty}\\
\leq& C\norm{v}_{W^{1,3}}\norm{\p_t u_\ep}^2_{H^1}\norm{u_\ep}_{H^2}\norm{u_\ep}_{H^3}\\
\leq &C(1+\norm{u_\ep}^2_{L^2}+\norm{\p_t u_\ep}^2_{H^1}+\norm{v}^2_{W^{1,3}})^4
\end{align*}
and
\begin{align*}
b=&|\int_{\Om}2\ep\<\n\p_tu_\ep,\n u_\ep\>\<u_\ep,\De \p_t u_\ep\>dx|\\
\leq& \frac{\ep}{8}\int_{\Om}|\De \p_tu_\ep|^2dx+C\ep\norm{\n u_\ep}^2_{L^\infty}\int_{\Om}|\n\p_t u_\ep|^2dx\\
\leq &\frac{\ep}{8}\int_{\Om}|\De \p_tu_\ep|^2dx+C\ep(1+\norm{u_\ep}^2_{L^2}+\norm{\p_t u_\ep}^2_{H^1}+\norm{v}^2_{W^{1,3}})^4.
\end{align*}

Next, we estimate $M_4$ as follows.

\begin{align*}
|M_4|=&|\int_{\Om}\<L, \De \p_t u_\ep\>dx|\\
=&|\int_{\Om}\<\p_tu_\ep\times(\De u_\ep+u_\ep \times\n_vu_\ep)+\ep|\n u_\ep|^2\p_tu_\ep+\ep\p_tu_\ep\times\n_vu_\ep,\De \p_t u_\ep\>dx|\\
=&K_1+K_2+K_3.
\end{align*}
Then to get a bound of $M_4$, we need to estimate $K_1$-$K_3$ step by steps. 

By applying the equation 
\[\tau_v(u_\ep)=\De u_\ep+|\n u_\ep|^2u_\ep+u_\ep\times \n_v u_\ep=\frac{1}{1+\ep^2}(\ep \p_t u_\ep-u_\ep\times \p_t u_\ep),\]
we can show 
\begin{align*}
|K_1|=&|\int_{\Om}\<\p_tu_\ep\times(\De u_\ep+u_\ep \times\n_vu_\ep),\De \p_t u_\ep\>dx|\\
\leq &\frac{1}{1+\ep^2}|\int_{\Om}\<\p_t u_\ep\times(\ep\p_t u_\ep-u_\ep\times \p_t u_\ep),\De \p_t u_\ep\>dx|\\
&+|\int_{\Om}\<\p_t u_\ep\times (|\n u_\ep|^2u_\ep),\De \p_t u_\ep\>dx|=c+d.
\end{align*}

This two terms $c$ and $d$ can be estimated as follows. Since
\begin{align*}
	\p_t u_\ep\times (u_\ep\times \n \p_t u_\ep)=&\<\p_t u_\ep,\n \p_t u_\ep\>u_\ep-\<\p_t u_\ep, u_\ep\>\n \p_t u_\ep\\
	=&\<\p_t u_\ep,\n \p_t u_\ep\>u_\ep,\\
\<u_\ep,\n \p_tu_\ep\>=&-\<\n u_\ep, \p_t u_\ep\>,
\end{align*}
then we have
\[\<\p_t u_\ep\times (u_\ep\times \n \p_t u_\ep),\n \p_t u_\ep\>=-\<\p_t u_\ep,\n \p_t u_\ep\>\<\n u_\ep, \p_t u_\ep\>.\]
Consequently, the term $c$ can be bounded as follows.

\begin{align*}
|c|=&\frac{1}{1+\ep^2}|\int_{\Om}\<\p_t u_\ep\times(\ep\p_t u_\ep-u_\ep\times \p_t u_\ep),\De \p_t u_\ep\>dx|\\
=&\frac{1}{1+\ep^2}|\int_{\Om}\<\p_t u_\ep\times\n (u_\ep\times \p_t u_\ep),\n \p_t u_\ep\>dx|\\
\leq &\int_{\Om}|\p_t u_\ep|^2|\n u_\ep||\n \p_t u_\ep|dx+\frac{1}{1+\ep^2}|\int_{\Om}\<\p_t u_\ep\times (u_\ep\times \n \p_t u_\ep),\n \p_t u_\ep\>dx|\\
\leq &C\int_{\Om}|\p_t u_\ep|^2|\n u_\ep||\n \p_t u_\ep|dx\leq C\norm{\n \p_t u_\ep}_{L^2}\norm{\n u_\ep}_{L^6}\norm{\p_t u_\ep}^2_{L^6}\\
\leq &C(1+\norm{u_\ep}^2_{H^2}+\norm{\p_tu_\ep}^2_{H^1})^2.
\end{align*}

On the other hand, a simple calculation shows that
\begin{align*}
|d|=&|\int_{\Om}\<\p_t u_\ep\times (|\n u_\ep|^2u_\ep),\De \p_t u_\ep\>dx|\\
=&|\int_{\Om}\<\p_t u_\ep\times \n(|\n u_\ep|^2u_\ep),\n \p_t u_\ep\>dx|\\
\leq &C\int_{\Om}|\p_t u_\ep||\n \p_t u_\ep|(|\n u_\ep|^3+|\n u_\ep||\n^2 u_\ep|)dx\\
\leq &C\norm{\n \p_t u_\ep}_{L^2}\norm{\p_t u_\ep}_{L^6}(\norm{\n u_\ep}_{L^\infty}\norm{\n u_\ep}^2_{L^6}+\norm{\n^2 u_\ep}_{L^6}\norm{\n u_\ep}_{L^6})\\
\leq &C(1+\norm{u_\ep}^2_{H^2}+\norm{\p_t u_\ep}^2_{H^1})^4.
\end{align*}
The above estimates for $c$ and $d$ lead to an upper bound of $K_1$
\[|K_1|\leq C(1+\norm{u_\ep}^2_{H^2}+\norm{\p_t u_\ep}^2_{H^1})^4.\]

The term $K_2$-$K_3$ can be estimated directly as follows.

\begin{align*}
|K_2|=&|\int_{\Om}\<\ep|\n u_\ep|^2\p_tu_\ep,\De \p_t u_\ep\>dx|\\
\leq& \frac{\ep}{8}\int_{\Om}|\De \p_t u_\ep|^2dx+C\ep\norm{\n u_\ep}^4_{L^6}\norm{\p_t u_\ep}^2_{L^6}\\
\leq &\frac{\ep}{8}\int_{\Om}|\De \p_t u_\ep|^2dx+C\ep(1+\norm{u_\ep}^2_{H^2}+\norm{\p_t u_\ep}^2_{H^1})^3.\\
|K_3|=&|\int_{\Om}\<\ep\p_tu_\ep\times\n_vu_\ep,\De \p_t u_\ep\>dx|\\
\leq &\frac{\ep}{8}\int_{\Om}|\De \p_t u_\ep|^2dx+C\ep\norm{v}^2_{L^6}\norm{\p_t u_\ep}^2_{L^6}\norm{\n u_\ep}^2_{L^6}\\
\leq &\frac{\ep}{8}\int_{\Om}|\De \p_t u_\ep|^2dx+C\ep(1+\norm{u_\ep}^2_{H^2}+\norm{\p_t u_\ep}^2_{H^1}+\norm{v}^2_{H^1})^3.
\end{align*}

Consequently, we can combine the above estimates for $K_1$-$K_3$ to show that
\[|M_4|\leq \frac{\ep}{4}\int_{\Om}|\De \p_t u_\ep|^2dx+C(1+\norm{u_\ep}^2_{H^2}+\norm{\p_t u_\ep}^2_{H^1}+\norm{v}^2_{H^1})^4.\]

It remains to estimate the term $M_5$. A simple calculation shows
\begin{align*}
|M_5|=&|\int_{\Om}\<K, \De \p_t u_\ep\>dx|\\
=&|\int_{\Om}\<-\n_{\p_tv}u_\ep+\ep u_\ep\times\n_{\p_tv}u_\ep, \De \p_t u_\ep\>dx|\\
\leq &|\int_{\Om}\<\n_{\p_tv}u_\ep, \De \p_t u_\ep\>dx|+|\int_{\Om}\<\ep u_\ep\times\n_{\p_tv}u_\ep, \De \p_t u_\ep\>dx|\\
=&e+h,
\end{align*}
where
\begin{align*}
e=&|\int_{\Om}\<\n_{\p_tv}u_\ep, \De \p_t u_\ep\>dx|=|\int_{\Om}\<\n(\n_{\p_tv}u_\ep), \n \p_t u_\ep\>dx|\\
\leq &\int_{\Om}|\n \p_t v||\n u_\ep||\n \p_t u_\ep|+|\p_t v||\n^2 u_\ep||\n\p_t u_\ep|dx\\
\leq &C\norm{\n \p_t v}_{L^2}\norm{\n \p_t u_\ep}_{L^2}\norm{\n u_\ep}_{L^\infty}+C\norm{\n^2u_\ep}_{L^6}\norm{\p_t v}_{L^3}\norm{\n \p_t u_\ep}\\
\leq &C(1+\norm{u_\ep}^2_{H^2}+\norm{\p_t u_\ep}^2_{H^1}+\norm{v}^2_{W^{1,3}})^4+C\norm{\p_tv}^2_{H^1},
\end{align*}
and
\begin{align*}
h=&|\int_{\Om}\<\ep u_\ep\times\n_{\p_tv}u_\ep, \De \p_t u_\ep\>dx|=\ep|\int_{\Om}\<\n( u_\ep\times\n_{\p_tv}u_\ep), \n \p_t u_\ep\>dx|\\
\leq &\ep\int_{\Om}|\n u_\ep|^2|\p_t v||\n \p_t u_\ep|dx+\ep|\int_{\Om}\< u_\ep\times\n(\n_{\p_tv}u_\ep), \n \p_t u_\ep\>dx|\\
\leq &C\ep(1+\norm{u_\ep}^2_{H^2}+\norm{\p_t u_\ep}^2_{H^1}+\norm{v}^2_{W^{1,3}})^4+C\ep\norm{\p_tv}^2_{H^1}.
\end{align*}

Therefore, by substituting the estimates of $M_1$-$M_5$ into \eqref{eq-H3}, we have
\begin{equation}\label{es-H3}
\begin{aligned}
&\frac{\p}{\p t}\int_{\Om}|\n \p_t u_\ep|^2dx+\ep\int_{\Om}|\De \p_tu_\ep|^2dx\\
\leq &C(1+\norm{u_\ep}^2_{H^2}+\norm{\p_t u_\ep}^2_{H^1}+\norm{v}^2_{W^{1,3}})^4+C(\norm{\p_t v}^2_{H^1}+\norm{v}^4_{L^\infty}+\norm{\n v}^2_{L^\infty}).
\end{aligned}
\end{equation}

Finally, we get a key uniform $H^3$-estimates for $u_\ep$ by combining the above inequalities \eqref{es-H1-new},\eqref{es-H2-new} with \eqref{es-H3}. We state this result as the following proposition. For the sake of convenience, we denote
\[G(u_\ep)=(1+\ep^2)\norm{u_\ep}^2_{H^2}+\norm{\p_tu_\ep}^2_{H^1}+1\]
and
\[g=\norm{\p_t v}^2_{H^1}+\norm{v}^4_{L^\infty}+\norm{\n v}^2_{L^\infty}.\]
\begin{prop}\label{Key-es}
Let $T_\ep$ and the solution $u_\ep$ be the same as that in Theorem \eqref{mth1}. Suppose that $v\in L^\infty(\Real^+,W^{1,3}(\Om))\cap C^0(\Real^+, H^1(\Om))\cap L^4(\Real^+, L^\infty(\Om))$, $\n v\in L^2(\Real^+, L^\infty(\Om))$, $\p_t v\in L^2(\Real^+, H^1(\Om))$, $\textnormal{\mbox{div}}(v)=0$ inside $\Om$ for any $t\in\Real^+$ and $\<v,\nu\>|_{\p\Om\times \Real^+}=0$. Then there holds
\begin{align*}
\frac{\p}{\p t}G(u_\ep)\leq C(G(u_\ep)+\norm{v}^2_{W^{1,3}})^4+Cg,
\end{align*}
for $0<t<T_\ep$.
	
Moreover, this inequality implies that there exists two constants $T_0$ and $C(T_0)$ depending only on $\norm{u_0}_{H^3(\Om)}$, $\norm{v}_{L^\infty(\Real^+,W^{1,3})}$ and $g$ such that
\[\sup_{0<t< \min \{T_0, T_\ep\}}(\norm{u_\ep}^2_{H^3}+\norm{\p_t u_\ep}^2_{H^1})\leq C(T_0).\]
\end{prop}
\begin{proof}
Since $u_\ep\in C^0([0,T_\ep), H^3(\Om))$ and $\p_t u_{\ep}\in C^0([0,T_\ep), H^1(\Om))$, we have
\[G(u_\ep)\in C^0[0,T_\ep).\]
Hence it is not difficult to show 
\[\norm{G(u_\ep)(0)}\leq C(\norm{u_0}^2_{H^3}+\norm{v(\cdot, 0)}^2_{H^1}+1)^3,\]
where we have applied the fact $v\in C^0(\Real^+, H^1(\Om)).$	
	
Then, by combining inequalities \eqref{es-H1-new}, \eqref{es-H2-new} and \eqref{es-H3}, we can show that $G(u_\ep)$ satisfies
\begin{equation*}
\begin{cases}
\frac{\p}{\p t}G(u_\ep)\leq C(G(u_\ep)+\norm{v}^2_{W^{1,3}})^4+Cg, \quad &t\in [0, \min\{T_0,T_\ep\}),\\[1ex]
G(u_\ep)(0)\leq C(\norm{u_0}^2_{H^3}+\norm{v(\cdot, 0)}^2_{H^1}+1)^3.
\end{cases}
\end{equation*}
	
Since $v\in L^\infty(\Real^+, W^{1,3}(\Om))$ and $g\in L^1(\Real^+)$, Lemma \ref{Gron-inq} implies that there exists two constants $T_0$ and $C(T_0)$ depending only on $\norm{u_0}_{H^3(\Om)}$, $\norm{v}_{L^\infty(\Real^+,W^{1,3})}$ and $g$ such that
\[\sup_{0\leq t< \min \{T_0, T_\ep\}}G(u_\ep)\leq C(T_0).\]
	
Therefore, the proof is completed.
\end{proof}
\subsection{Local regular solutions of the incompressible Schr\"odinger flow}\label{s: proof-main-thm1}
In this part, we show our main result on the existence of local regular solutions to \eqref{eq-ISMF}, i.e. Theorem \ref{mth2}. We will only provide a brief outline of the proof for Theorem \ref{mth2}, since the arguments are almost identical to those used in the proof of the existence of local regular solutions to the Schr\"odinger flow into $\U^2$ in \cite{CW1}, once we have obtained a uniform $H^3$-estimates of the approximation solutions $u_\ep$.

\medskip
\begin{proof}[\bf{The proof of Theorem \ref{mth2}}]
Our proof is divided into two steps.

\medskip
\noindent\emph{Step 1: A uniform lower bound of $T_\ep$.}\

Without lose of generality, we assume that $T^\prime_\ep$ is the maximal value such that the estimates \eqref{es-app-sol} hold with $T^\prime_\ep=T_\ep$. Then, we claim that $T^\prime_\ep>T_0$.

On the contrary, if $T_0<T^\prime_\ep$, then Proposition \ref{Key-es} implies that the solution $u_\ep$ satisfies
\[\sup_{0\leq t< T_\ep}(\norm{u_\ep}^2_{H^3}+\norm{\p_t u_\ep}^2_{H^1})\leq C(T_0).\]
By applying similar arguments as that in \cite{CW1}, we can utilize the above estimate for $u_\ep$ to show that the estimates \eqref{es-app-sol} hold true for any $0<T\leq T^\prime_\ep$, which leads to a contradiction with the definition of $T^\prime_\ep$.

\medskip
\noindent\emph{Step 2: Local regular solutions of the incompressible Schr\"odinger flow.}\

The result in step 1 tells us that
\[\sup_{0\leq t\leq T_0}(\norm{u_\ep}^2_{H^3}+\norm{\p_t u_\ep}^2_{H^1})\leq C(T_0),\]
where $0<T_0<T^\prime_\ep$.

Consequently, with the above uniform $H^3$- estimate for $u_\ep$ at hand, we then apply Lemma \ref{A-S} to demonstrate that there exists a subsequence of $\{u_\ep\}$ such that which converges to a local regular solution $u$ to the problem \eqref{eq-ISMF} as $\ep\to 0$. Moreover, this local regular solution $u$ satisfies the estimate \eqref{es-mth1}.
\end{proof}

\medskip\medskip
\noindent {\it\bf{Acknowledgements}}: The author B. Chen is supported partially by NSFC (Grant No. 12301074) and Guangzhou Basic and Applied Basic Research Foundation (Grant No. 2024A04J3637), the author Y.D. Wang is supported partially by NSFC (Grant No.11971400) and National key Research and Development projects of China (Grant No. 2020YFA0712500).

\medskip
\section*{Statements and Declarations}
\noindent {\it\bf{Competing interests}}: The authors declare that no conflict of interest exists in this article.

\end{document}